\documentclass[pdftex,a4paper,oneside]{article}
\usepackage{amsmath,amssymb,latexsym,fancyhdr,amsthm,natbib}

\usepackage[english]{babel}
\usepackage[utf8]{inputenc} 
\usepackage{bm}
\usepackage{color} 
\usepackage[babel]{csquotes} 
\usepackage{tikz-cd}
\usepackage{tikz}
\usepackage{bussproofs}
\usepackage{stmaryrd} 
\usepackage{epigraph} 
\usepackage{graphicx} 
\usepackage{rotating} 
\usepackage{enumitem}

\renewcommand{\phi}{\varphi}

\newbox\gnBoxA
\newdimen\gnCornerHgt
\setbox\gnBoxA=\hbox{$\ulcorner$}
\global\gnCornerHgt=\ht\gnBoxA
\newdimen\gnArgHgt
\newcommand{\gn}[2]{%
	\setbox\gnBoxA=\hbox{$#1$}%
	\gnArgHgt=\ht\gnBoxA%
	\ifnum     \gnArgHgt<\gnCornerHgt \gnArgHgt=0pt%
	\else \advance \gnArgHgt by -\gnCornerHgt%
	\fi \raise\gnArgHgt\hbox{$\ulcorner$} \box\gnBoxA %
	\raise\gnArgHgt\hbox{$\urcorner^{#2}$}}

\newcommand{\concat}{%
	\mathord{
		\mathchoice
		{\raisebox{1ex}{\scalebox{.7}{$\frown$}}}
		{\raisebox{1ex}{\scalebox{.7}{$\frown$}}}
		{\raisebox{.7ex}{\scalebox{.5}{$\frown$}}}
		{\raisebox{.7ex}{\scalebox{.5}{$\frown$}}}
	}
}

\DeclareMathSymbol{\mlq}{\mathord}{operators}{``}
\DeclareMathSymbol{\mrq}{\mathord}{operators}{`'}

\newcounter{bew}
\newcommand\bwz[2]{&(\arabic{bew})&&#1&&\text{#2}\stepcounter{bew}}
\newenvironment{nbeweis}{\setcounter{bew}{1}\csname align*\endcsname}{\csname endalign*\endcsname}

\renewcommand{\restriction}{\mathord{\upharpoonright}}

\DeclareMathOperator{\dom}{dom} 

\newcommand{\Forall}{\raisebox{0.15em}{\scalebox{0.72}{$\backslash$}}\hspace{-0.18em}\forall}
\newcommand{\Exists}{\exists\hspace{-0.45em}\exists}

\newcommand{\Land}{\;\raisebox{-0.12em}{\begin{turn}{90}$\eqslantgtr$\end{turn}}\;}

\newcommand\hash{\mathbin{\mathpalette\xhash\relax}}
\newcommand{\xhash}[2]{\ooalign{%
		$#1\xxhash{#1}{-45}$\cr
		$#1\xxhash{#1}{45}$\cr
	}%
}
\newcommand{\xxhash}[2]{\rotatebox[origin=c]{#2}{$#1\parallel$}}

\makeatletter
\newcommand{\ogeneric}[2][0.7]{%
	\vphantom{\oplus}\mathpalette\o@generic{{#1}{#2}}%
}
\newcommand{\o@generic}[2]{\o@@generic#1#2}
\newcommand{\o@@generic}[3]{%
	\begingroup
	\sbox\z@{$\m@th#1\oplus$}%
	\dimen@=\dimexpr\ht\z@+\dp\z@\relax
	\savebox\tw@[\totalheight]{$\m@th#1\bigcirc$}%
	\makebox[\wd\z@]{%
		\ooalign{%
			$#1\vcenter{\hbox{\resizebox{\dimen@}{!}{\usebox\tw@}}}$\cr
			\hidewidth
			$#1\vcenter{\hbox{\resizebox{#2\dimen@}{!}{$#1\vphantom{\oplus}{#3}$}}}$%
			\hidewidth
			\cr
		}%
	}%
	\endgroup
}
\makeatother

\newcommand{\opreceq}{\mathrel{\ogeneric[0.6]{\preceq}}}
\newcommand{\oaste}{\mathbin{\ogeneric[0.6]{\ast}}}

\newcommand{\oefunc}{\mathrel{\ogeneric[0.6]{e}}}
\newcommand{\oewedge}{\mathrel{\ogeneric[0.6]{\underline{\wedge}}}}

\theoremstyle{plain}
\newtheorem{thm}{Theorem}[section]
\newtheorem{lem}[thm]{Lemma}
\newtheorem{cor}[thm]{Corollary}

\theoremstyle{definition}
\newtheorem{definition}[thm]{Definition}

\title{On the Invariance of Gödel’s Second Theorem with regard to Numberings}

\author{Balthasar Grabmayr\thanks{This is the penultimate draft of a paper forthcoming in \textit{The Review of Symbolic Logic}. Please cite the final version.}\\Humboldt University of Berlin}

\date{}

\begin{document}
	
	\maketitle
	
	\epigraph{Let there be assigned to any term its symbolic number, to be used in calculation as the term itself is used in reasoning.}{(Leibniz, \citeyear{Leibniz1679}/\citeyear{Parkinson1966}, p.~17)}
	
	\begin{abstract}
		The prevalent interpretation of Gödel's Second Theorem states that a sufficiently adequate and consistent theory does not prove its consistency. It is however not entirely clear how to justify this informal reading, as the formulation of the underlying mathematical theorem depends on several arbitrary formalisation choices. In this paper I examine the theorem's dependency regarding Gödel numberings. I introduce \textit{deviant} numberings, yielding provability predicates satisfying Löb's conditions, which result in provable consistency sentences. According to the main result of this paper however, these \enquote{counterexamples} do not refute the theorem's prevalent interpretation, since once a natural class of \textit{admissible} numberings is singled out, invariance is maintained.
	\end{abstract}
	
	\section{Introduction}
	
	According to the prevalent philosophical interpretation of Gödel’s Second Theorem, a sufficiently adequate and consistent theory $T$ does not prove its consistency. Upon further examination however, it is not entirely clear how to justify this informal reading, as the formulation of the underlying mathematical theorem depends on several arbitrary implementation choices. \cite{Visser2011} for instance locates three sources of indeterminacy in the formalisation of the consistency statement for a theory~$T$:
	\begin{enumerate}[noitemsep,label=\Roman*.]
		\item the choice of a proof system;
		\item the choice of a Gödel numbering;
		\item the choice of a specific formula representing the axiom set of $T$.
	\end{enumerate}
	
	By affirming the theorem's prevalent interpretation we (implicitly) rely on its \textit{invariance} regarding these formalisation choices. That is, we believe that any admissible choice yields a formal consistency sentence which is unprovable in $T$.
	
	This belief is only partially supported by metamathematical results. While employing fixed choices for (I) and (II), \cite{Feferman1960} provides the invariance of Gödel's theorem with respect to~(III), i.e., regarding a certain class of formul{\ae} which represent $T$'s axioms. Visser's (2011) approach rests on fixed choices for (II) and (III) but is independent of (I). The aim of the present paper is to examine the theorem's dependency regarding~(II), which to the best of my knowledge has hitherto not been treated in the literature.
	
	I follow the popular custom of basing the notion of provability on Löb's conditions. It is widely believed that these conditions are already sufficient to derive the unprovability of the resulting consistency sentences. This is however not the case, as I will show in this paper, since Gödel's theorem based on Löb's conditions in fact depends on the underlying numbering. Consider for instance a numbering which assigns even numbers to theorems and odd numbers to all other expressions of a given language.
	The formula $\exists y \ x = \overline{2} \times y$ then serves as a binumeration of the codes of theorems, which in addition satisfies Löb's conditions. However, the resulting consistency sentence
	is provable.
	
	A similar construction shows that the set of true sentences is arithmetically definable, once such deviant numberings are employed. Thus also the prevalent interpretation of Tarski’s theorem relies on a more careful choice of the underlying Gödel numbering.
	
	I argue that these \enquote{counterexamples} do not threaten these theorems’ prevalent interpretations, since the employed numberings are inadmissible choices in the formalisation process. In fact, I show that once natural classes of \textit{admissible} numberings are singled out, invariance of these theorems holds.
	
	This work can be seen to provide further justification to the prevalent philosophical readings of these limitative results as sketched above. For instance, \cite{Detlefsen1986} famously argues that only once \textit{every} sentence expressing \mbox{$T$-consistency} is shown to be unprovable in $T$, the prevalent philosophical reading of Gödel's Second Theorem can be maintained. Under the assumption that Löb's conditions are necessary for a formula to express $T$-provability,\footnote{\label{ftn:lob}Löb's conditions indeed play an important role in the literature as necessary conditions for a formula to express $T$-provability. \cite{Dyson1991}, for instance, takes these conditions as \enquote{the axiomatic formulation of minimal requirements that a meaningful concept of theoremhood ought to satisfy} (p.~256). A similar view is reported in \citep{Bezboruah1976} and \citep{Kennedy}. For an extensive philosophical justification of the adequacy of Löb's conditions the reader is referred to \citep{Auerbach1992} and \citep{Roeper2003}. For arguments against their adequacy see \citep{Detlefsen1986, Detlefsen2001}.} the result obtained in this paper serves as a further step towards this goal. Namely, it establishes that no matter which admissible numbering is employed in the formulation of Löb's conditions, every formula expressing \mbox{$T$-provability} results in a consistency sentence which is not provable in $T$. A fully satisfactory account along these lines clearly also requires consideration of other aspects of formalisation not treated here, such as the different ways consistency sentences are constructed from formul{\ae} expressing provability.\footnote{For a study of this aspect of formalisation the reader is referred to \citep{Kurahashi2019}.}
	
	The plan of this paper is as follows. Section 2 provides an introduction to the popular version of Gödel's Second Theorem based on Löb's conditions as well as \enquote{counterexamples} thereof which are based on deviant numberings. I discuss the notion of a numbering's admissibility in Section 3, where I argue that computability is a necessary condition for the admissibility of a numbering. A precise and robust notion of computability then allows for a (meta-)mathematical restatement of the claim that Gödel's theorem is invariant regarding admissible numberings, which is proved in Section 4. Finally, Section 5 contains some concluding remarks.

	\section{Dependency Results}
	
	\subsection{Technical Preliminaries}
	\label{subsection:prelimin}
	Let the language $\mathcal{L}$ be given by the signature $\sigma(\mathcal{L}) = \{ \mlq \mathsf{0} \mrq, \mlq \mathsf{S} \mrq , \mlq  + \mrq , \mlq \times \mrq, \mlq = \mrq , \mlq \neg \mrq , \mlq \wedge \mrq , \mlq \forall \mrq  \}$. Consider the alphabet $A$ which consists of $\sigma(\mathcal{L})$ together with the symbols $\mlq \mathsf{v} \mrq$ , $\mlq \prime \mrq$ , $\mlq ( \mrq$ and $\mlq ) \mrq$. Let $A^\ast$ denote the set of finite strings over $A$ (without the empty string), and let $\ast$ and $\equiv$ denote the concatenation operation and the equality relation on $A^\ast$ respectively. For better readability, we often omit the use of quotation symbols and the concatenaton operation. For instance, for $s,t \in A^\ast$ we write $s  \prime$ instead of $s \ast \mlq \prime \mrq$ and $(\mathsf{S} t)$ instead of $ \mlq (\mathsf{S} \mrq \ast t \ast \mlq ) \mrq $, etc. We take an $\mathcal{L}$-expression to be any element of $A^\ast$. The set $\mathit{Var}$ of $\mathcal{L}$-variables is generated from $\mathsf{v}$ by appending a finite number of strokes. The sets $\mathit{Term}, \mathit{Fml} \subseteq A^\ast$ contain the well-formed $\mathcal{L}$-terms and $\mathcal{L}$-formul{\ae} given in infix notation respectively.\footnote{The results of this paper also apply to any language $\mathcal{L}^+ \supseteq \mathcal{L}$ and finite alphabet~$A^+ \supseteq A$ such that the well-formed $\mathcal{L}^+$-expressions are contained in $(A^+)^\ast$.} Any used symbol not contained in $\mathcal{L}$ is to be understood as the usual abbreviation by symbols of $\mathcal{L}$. For instance, we define $x \leq y \equiv \exists z \ z + x = y$, where $x,y,z \in \mathit{Var}$ and $\exists$ is defined by means of $\forall$ and $\neg$ as usual. Since the language $\mathcal{L}$ is fixed, expressions, terms, formul{\ae} and sentences are throughout the paper implicitly assumed to belong to~$\mathcal{L}$. For the sake of brevity, we sometimes use the symbols ${\Rightarrow}$, ${\Land}$, ${\Forall}$ and ${\Exists}$ as abbreviations of the meta-theoretical connectives \enquote{implies}, \enquote{and}, \enquote{for all} and \enquote{there is} respectively.
	
	Corresponding to their alphabetical expressions, the following \enquote{constructor} operations on strings play an important role throughout this paper:
	\begin{align*}
	\underline{\prime} & \colon \mathit{Var} \to \mathit{Var}, \text{ given by } x \mapsto x  \prime ;\\
	\underline{\mathsf{S}} & \colon \mathit{Term} \to \mathit{Term}, \text{ given by } t \mapsto (\mathsf{S} t);\\
	\underline{+} & \colon \mathit{Term} \times \mathit{Term} \to \mathit{Term}, \text{ given by } (s,t) \mapsto (s + t);\\
	\underline{\times} & \colon \mathit{Term} \times \mathit{Term} \to \mathit{Term}, \text{ given by } (s,t) \mapsto (s \times t);\\
	\underline{=} & \colon \mathit{Term} \times \mathit{Term} \to \mathit{Fml}, \text{ given by } (s,t) \mapsto (s = t);\\
	\underline{\neg} & \colon \mathit{Fml} \to \mathit{Fml}, \text{ given by } \phi \mapsto (\neg \phi);\\
	\underline{\wedge} & \colon \mathit{Fml} \times \mathit{Fml} \to \mathit{Fml}, \text{ given by } (\phi,\psi) \mapsto (\phi \wedge \psi);\\
	\underline{\forall} & \colon \mathit{Var} \times \mathit{Fml} \to \mathit{Fml}, \text{ given by } (x,\phi) \mapsto ( \forall x  \phi );
	\end{align*}
	
	By slight abuse of notation, both the standard interpretation of $\mathcal{L}$ and its domain is denoted by $\mathbb{N}$. As usual, $\overline{n}$ denotes the (standard) numeral of the number $n$. Let $\phi$ be an expression and let $\alpha$ be a numbering of $A^\ast$, i.e., $\alpha \colon A^\ast \to \mathbb{N}$ is injective. We write $\gn{\phi}{\alpha}$ for the numeral of the $\alpha$-code $\alpha(\phi)$ of $\phi$ and we write $\gn{\phi}{}$ if no specific numbering is indicated.
	
	We call a function $\alpha \colon B \to \mathbb{N}$ recursive, if $B \subseteq \mathbb{N}$ is decidable and there exists a partial recursive function $\alpha' \colon \mathbb{N} \to \mathbb{N}$ with $\dom(\alpha') = B$ and $\alpha' \restriction B = \alpha$. Let $\alpha_1$ and $\alpha_2$ be numberings of given sets $S_1$ and $S_2$ respectively. We call a function $f \colon S_1 \to S_2 $ recursive relative to $\langle \alpha_1, \alpha_2 \rangle$ if $\alpha_2 \circ f \circ \alpha_1^{-1}$ is recursive. If moreover $\alpha_1 = \alpha_2$, we also say that $f$ is recursive relative to $\alpha_1$.
	
	In this paper we consider consistent and \enquote{sufficiently adequate} $\mathcal{L}$-theories. We take any \enquote{sufficiently adequate} $\mathcal{L}$-theory to be recursively enumerable and to extend the Tarski-Mostowski-Robinson theory~$\mathsf{R}$, which is given by the following axiom schemata:\footnote{See \citep[p.\ 53]{TarskiMostowskiRobinson}. For further information about $\mathsf{R}$ the reader may also consult \citep{Visser2014}.}
	\begin{enumerate}[noitemsep,label=$\mathsf{R\arabic*}$.]
		\item $\overline{m} + \overline{n} = \overline{m + n}$
		\item $\overline{m} \times \overline{n} = \overline{m \cdot n}$
		\item $\overline{m} \neq \overline{n}, \text{ for } m \neq n$
		\item $x \leq \overline{n} \rightarrow \bigvee_{k \leq n} x = \overline{k}$
		\item $x \leq \overline{n} \vee \overline{n} \leq x$
	\end{enumerate}
	
	\subsection{Löb's Conditions}
	Perhaps the most popular approach to abstracting away from specific formalisation choices regarding consistency sentences rests on the so-called Hilbert-Bernays-Löb derivability conditions, or for short: Löb's conditions, which are defined as follows. A formula $\mathsf{Pr}(x)$ satisfies Löb's conditions for $T$, in short: Löb($T$), if for all sentences $\phi$ and $\psi$:
	\begin{description}[noitemsep]
		\item[Löb1] $T \vdash \phi \implies T \vdash \mathsf{Pr}(\gn{\phi}{})$;
		\item[Löb2] $T \vdash \mathsf{Pr}(\gn{\phi}{}) \wedge \mathsf{Pr}(\gn{\phi \rightarrow \psi}{}) \rightarrow \mathsf{Pr}(\gn{\psi}{}) $;
		\item[Löb3]
		$T \vdash \mathsf{Pr}(\gn{\phi}{}) \rightarrow \mathsf{Pr}(\gn{\mathsf{Pr}(\gn{\phi}{})  }{})$.
	\end{description}
	Let $\bot$ be some fixed $T$-contradiction, such as $\mathsf{0} = \mathsf{S0}$. This approach results in the following famous generalisation of Gödel's Second Theorem due to \cite{HilbertBernays1939} and \cite{Lob1955}, which constitutes the starting point of this study:
	\begin{thm}
		\label{standardGodel2}
		$T \not\vdash \neg \mathsf{Pr}(\gn{\bot}{})$, for all formul{\ae} $\mathsf{Pr}(x)$ satisfying \textnormal{Löb($T$)}.
	\end{thm}
	
	Indeed, this version of Gödel's theorem neither resorts to the underlying proof system nor to the specific way $T$-provability and $T$'s axiom set are formalised.\footnote{This is certainly true of the theorem's formulation. Whether Löb's conditions can be \textit{justified} independently of these choices is however a more subtle issue and depends on the specific strategy of justification (see for instance \citep{Detlefsen2001} and \citep{Visser2016}).} However, its formulation employs a specific and arbitrarily chosen Gödel numbering. In order to also abstract away from this source of indeterminacy, we first make the numbering used in the formulation of Löb's conditions explicit, where \textit{prima facie} any injective function $\alpha \colon A^\ast \to \mathbb{N}$ qualifies as a numbering. We then take a formula $\mathsf{Pr}^\alpha(x)$ to satisfy Löb's conditions \textit{relative to $\alpha$} for $T$, in short: Löb($T, \alpha$), if for all sentences $\phi$ and $\psi$:
	\begin{description}[noitemsep]
		\item[$\text{Löb1}^{\bm{\alpha}}$] $T \vdash \phi \implies T \vdash \mathsf{Pr}^\alpha(\gn{\phi}{\alpha})$;
		\item[$\text{Löb2}^{\bm{\alpha}}$] $T \vdash \mathsf{Pr}^\alpha(\gn{\phi}{\alpha}) \wedge \mathsf{Pr}^\alpha(\gn{\phi \rightarrow \psi}{\alpha}) \rightarrow \mathsf{Pr}^\alpha(\gn{\psi}{\alpha}) $;
		\item[$\text{Löb3}^{\bm{\alpha}}$]
		$T \vdash \mathsf{Pr}^\alpha(\gn{\phi}{\alpha}) \rightarrow \mathsf{Pr}^\alpha(\gn{\mathsf{Pr}^\alpha(\gn{\phi}{\alpha})  }{\alpha}) $.
	\end{description}
	
	It is widely believed that Löb's conditions are sufficient to allow for the \enquote{modal reasoning} involved in the proof of Löb's theorem, and thus in particular, of Theorem~\ref{standardGodel2} (for such a proof see e.g.\ \cite[p. 230]{Smith}). As this \enquote{modal reasoning} does not resort to the underlying numbering, one might expect Theorem~\ref{standardGodel2} to be invariant in the following sense: whatever numbering $\alpha$ is chosen, $T \not\vdash \neg \mathsf{Pr}(\gn{\bot}{\alpha})$ for every formula $\mathsf{Pr}^{\alpha}(x)$ satisfying~Löb($T, \alpha$).
	
	\subsection{Deviant Numberings}
	\label{subsection:deviant}
	
	I now show that this invariance claim is false, as there exist highly artificial and deviant numberings. To start with, for any given set $S \subseteq A^\ast$ of expressions an injective function $\alpha \colon A^\ast \to \mathbb{N}$ can be defined such that the set of $\alpha$-codes of $S$ is $\Delta^0_0$-binumerable.\footnote{If no theory is specified, this should be read as \enquote{(bi-)numerable in $\mathsf{R}$} (by a $\Delta^0_0$-formula).} For instance, define $\alpha$ such that
	\begin{equation*}
	\alpha(\phi) \text{ is }
	\begin{cases}
	\text{even} & \text{ if $\phi \in S$,}\\
	\text{odd} & \text{ if $\phi \not\in S$. }
	\end{cases}
	\end{equation*}
	This can simply be done by enumerating both $S$ and $A^\ast \setminus S$ (without repetitions) and subsequently assigning the number $2 \cdot k$ to the $k$-th element of the enumeration of $S$ and the number $2 \cdot k +1$ to the $k$-th element of the enumeration of $A^\ast \setminus S$. Clearly, these enumerations are possibly non-effective.
	
	Let $S := T^{\vdash}$ be the deductive closure of $T$ and let $\alpha$ be defined as above, such that the set of $\alpha$-codes of $T$-theorems are exactly the even numbers. The formula $\mathsf{Pr}^\alpha(x) \equiv \exists y < x \ x = \overline{2} \times y$ then satisfies Löb($T, \alpha$) and we have $T \vdash \neg \mathsf{Pr}^\alpha(\gn{\bot}{})$ by the $\Sigma^0_1$-completeness of $T$, in violation of the above invariance claim. The reason for this is that in addition to Löb's conditions,  the fixed point construction is also a crucial ingredient in the \enquote{modal reasoning} towards Löb's theorem. And indeed, the Diagonal Lemma fails with regard to $\alpha$, i.e., there is no ($\alpha$-)fixpoint of $\neg \mathsf{Pr}^\alpha(x)$. To see this, assume that there exists $G$ such that $T \vdash \neg \mathsf{Pr}^\alpha(\gn{G}{\alpha}) \leftrightarrow G$, which implies $T \not\vdash G$. But then $T \vdash \neg \mathsf{Pr}^\alpha(\gn{G}{\alpha})$ and therefore $T \vdash G$, yielding a contradiction.
	
	Similarly, setting $S := \mathsf{Th}(\mathbb{N})$ yields a numbering $\beta$ such that $\exists y < x \ x = \overline{2} \times y$ binumerates $\{ \beta(\phi) \mid \mathbb{N} \models \phi \}$, thus violating (both the semantic and the syntactic version of) Tarski's Theorem.
	
	While $\alpha$ and $\beta$ appear to be highly deviant numberings which resist any arithmetisation of the usual syntactic properties, less artificial numberings can be constructed for certain theories $T \supseteq \mathsf{R}$ which still violate Gödel's Second Theorem and the two versions of Tarski's Theorem respectively.\footnote{I would like to thank Albert Visser for raising this point and suggesting the initial idea underlying the construction of $\delta$.} For instance, based on a more careful construction, both $\alpha$ and $\beta$ can be taken to be monotonic (see $\zeta$ in Appendix~6). That is, the Gödel number of a string is larger than the Gödel numbers of its substrings. There also exists a numbering $\delta$ such that the functions $\underline{\mathsf{S}}$, $\underline{+}$, $\underline{\times}$, $\underline{\prime}$, $\underline{=}$, $\underline{\wedge}$ and $\underline{\forall}$ are recursive relative to $\delta$, yet the set of $T$-theorems is binumerable by a $\Delta^0_0(\mathsf{exp})$-formula $\mathsf{Pr}^{\delta}(x)$ (see Appendix~6).\footnote{The complexity class $\Delta^0_0(\mathsf{exp})$ contains the bounded formul{\ae} of $\mathcal{L}$ augmented with a unary function symbol $\mathsf{exp}$ for the exponentiation function $\lambda n.2^n$. We can also speak of $\Delta^0_0(\mathsf{exp})$-formul{\ae} in the context of $\mathcal{L}$, since every formula of $\Delta^0_0(\mathsf{exp})$ can be transformed into an $\mathcal{L}$-formula. This proceeds by the usual term elimination algorithm which replaces each occurrence of $\mathsf{exp}$ by an $\mathcal{L}$-formula.} In particular, $\mathsf{Pr}^{\delta}(x)$ once again satisfies Löb($T, \delta$), but $T \vdash \neg \mathsf{Pr}^{\delta}(\gn{\bot}{\delta})$.
	
	There exists moreover a numbering $\eta$ such that all the functions $\underline{\mathsf{S}}$, $\underline{+}$, $\underline{\times}$, $\underline{\prime}$, $\underline{=}$, $\underline{\neg}$, $\underline{\wedge}$ and $\underline{\rightarrow}$ are recursive relative to $\eta$ and the set $\{ \eta(\phi) \mid \mathcal{N} \models \phi \}$ of true sentences is binumerable by a $\Delta^0_0(\mathsf{exp})$-formula $\mathsf{Pr}^{\eta}(x)$ (see Appendix~6). In particular, the formula $\mathsf{Pr}^{\eta}(x)$ satisfies Löb($T, \eta$) for sound theories $T \supseteq \mathsf{R}$, but $T \vdash \neg \mathsf{Pr}^{\eta}(\gn{\bot}{\eta})$.
	
	Both numberings $\delta$ and $\eta$ thus yield deviant results, even though they allow for the arithmetisation of a large portion of syntactic properties and operations. For instance, the $\mathsf{num}$-function (mapping numbers to their standard numerals) and the substitution function for terms as well as for atomic formul{\ae} can be binumerated relative to both numberings. However, since $\underline{\neg}$ and $\underline{\forall}$ are not recursive relative to $\delta$ and $\eta$ respectively, the arithmetisation of the substitution function for (complex) formul{\ae} fails for both numberings. This can be seen as the reason why the Diagonal Lemma does not hold when employing these numberings, since its proof crucially relies on the arithmetisation of the substitution function.\\
	
	In what follows I will argue that these results do not refute Gödel and Tarski's classical theorems, since the underlying numberings are inadmissible for the contexts they are employed in.\footnote{\label{ftn:meaningpost}One could also argue that the employed formul{\ae} do not adequately express \mbox{$T$-provability} by resorting to their deviant \textit{quantificational structure}. For instance, $\exists y \ x = \overline{2} \times y$ can be argued to be \textit{intensionally incorrect}, since it does not structurally resemble the usual description of a meta-theoretical provability predicate, given by \enquote{$\Exists y \colon y \text{ is a proof of x}$}, where \enquote{$y$ is proof of $x$} is in turn described by existential quantification over finite sequences which contain axioms and are closed under rules of inference, etc. While this is surely a reasonable response, it is notoriously difficult to make this \textit{resemblance criterion} for expressibility precise \citep[pp.~676f.]{HalbachVisser1}. What is more, in this paper we are concerned with the invariance of Theorem~\ref{standardGodel2}, whose formulation does not resort to any such structural features of the underlying provability predicates.} Our initial na\"ive approach to numberings is thus not tenable and has to be replaced by a refined account based on \textit{admissible} numberings, ruling out such deviant constructions. This is precisely the aim of the next section.
	
	\section{Admissibility of Numberings}
	\label{section:admissibilityofnumberings}
	
	It is notoriously difficult to make the informal notion of the admissibility of a numbering precise. However, for the purposes of this paper it is sufficient to isolate a necessary condition of admissibility, namely, the computability of a numbering. More specifically, computable numberings will here be characterised by \textit{computable simulations} of the structure of strings together with the concatenation operation. Once the notion of computable simulation is made precise, I will discuss its adequacy and compare it to three other proposals found in the literature. I will conclude this section by considering other approaches to the admissibility of numberings based on compositionality and monotonicity rather than computational considerations.
	
	\subsection{Computable Simulations}
	\label{subsection:computablesimulations}
	
	The reader is reminded that the set of expressions is given by the set $A^\ast$ of strings over the alphabet  $A$. Let $\alpha \colon A^\ast \to \mathbb{N}$ be any injective function and let $G := \alpha(A^\ast)$ be the set of Gödel numbers, or $\alpha$-codes, of $A^\ast$. For any operation $e \colon (A^\ast)^k \to A^\ast$, we call the function $\oefunc \colon G^k \to G$ given by
	\[
	{\oefunc}(n_1,\ldots,n_k) = \alpha(e(\alpha^{-1}(n_1),\ldots,\alpha^{-1}(n_k)  )  ),
	\]
	the ($\alpha$-)tracking function of $e$. To illustrate, $\oefunc$ is defined such that the diagram
	\begin{equation}
	\label{commutdiagram}
	\begin{tikzcd}
	A^\ast \times \cdots \times A^\ast \arrow[r, "e"] \arrow[d, "\alpha \times \cdots \times \alpha"] & A^\ast \arrow[d, "\alpha"] \\
	G \times \cdots \times G \arrow[r, "\oefunc"] & G
	\end{tikzcd}
	\end{equation}
	commutes, where $\alpha \times \cdots \times \alpha(a_1, \ldots, a_k) := (\alpha(a_1),\ldots, \alpha(a_k))$. The function $\oefunc$ can thus be seen to simulate $e$ on the set of $\alpha$-codes. For this reason we also sometimes call $\oefunc$ the $e$-simulating function (on $G$).
	
	Intrinsic to the structure of strings is the concatenation operation $\ast$, which allows to generate any string of $A^\ast$ from elements of $A$. The operation $\oaste$ can then be seen to simulate the concatenation operator on the set $G$ of Gödel numbers. It is easy to check that the associativity of $\ast$ is inherited by $\oaste$. Moreover, as in the case of strings over $A$, each $\alpha$-code in~$G$ can be finitely generated from the set $\alpha(A)$ of $\alpha$-codes of~$A$ by the $\ast$-simulating operation $\oaste$. For instance, for any string $s \in A^\ast$ there are alphabetical expressions $s_1, \ldots, s_n \in A$ such that
	\[
	s \equiv s_1 \ast \cdots \ast s_n.
	\]
	The $\alpha$-code of $s$ is then simply the result of \enquote{\mbox{$\oaste$-concatenating}} the $\alpha$-codes of the alphabetical expressions $s_i$, i.e.,
	\[
	\alpha(s) = \alpha(s_1) \oaste \cdots \oaste \alpha(s_n).
	\]
	This situation can be naturally described in algebraic terms, by conceiving of $A^\ast$ together with $\ast$ as a semigroup. The above considerations then show that also $(G,\oaste)$ forms a semigroup and $\alpha$ is a (semigroup) isomorphism.
	
	Thus far, no requirements have been imposed on the numbering $\alpha$. That $\alpha$ \textit{computably simulates} $A^\ast$ can now be informally expressed by (i) requiring that for each number it is \textit{decidable} whether or not it $\alpha$-codes an expression, and by (ii) requiring that Gödel numbers are \textit{mechanically} generated from the $\alpha$-codes of alphabetic expressions in $A$. Since $G$ is a set of natural numbers and $\oaste$ is an arithmetical operation which generates $G$ from $\alpha(A)$, both requirements can be explicated via the Church-Turing Thesis as follows:
	
	\begin{definition}
		\label{defi:compnumbering}
		A numbering $\alpha \colon A^\ast \to \mathbb{N}$ is a computable simulation \textnormal{(}of $A^\ast$\textnormal{)}, if $G$ is decidable and $\oaste$ is recursive, where $G$ and $\oaste$ are given as above.
	\end{definition}
	
	This notion of a computable simulation can be traced back to \citep[chapter IV]{Montague1957} as well as to the foundational texts \citep{Rabin1960} and \citep{Malcev1961} of the field of computable algebra.\footnote{The definition typically used in computable algebra does not however require numberings to be functional. Indeed, as opposed to injectivity, the functionality of a numbering is not essential as long as its \enquote{kernel} is decidable. However, in order to simplify the presentation I follow the common practice in metamathematics and require the functionality of numberings.} The reader who questions the need for such a notion is reminded that the explication of a numbering's computability via the Church-Turing Thesis is not entirely straight-forward. Turing machines for instance only operate over strings. A sensible application of Turing-computability to numberings thus requires a representation of their co-domain $\mathbb{N}$ by strings, which in general yields different sets of computable numberings \citep{Montague1960}. The task of singling out admissible or natural representations has proven to be fraught with difficulty, purportedly even resulting in a circular conceptual analysis of (numerical) computability, see \citep{Boker2008,Boker2010}, \citep{Quinon2018}, \citep{Rescorla2007} and \citep{Shapiro1982,Shapiro2017}.
	
	As will be shown in Section~\ref{subsection:equivofnumberings}, computable simulations do not depend on such arbitrary representational choices, thus providing an adequate explication of computability for our purposes.
	
	\subsection{The Epistemic Role of Numberings}
	\label{subsection:epistemicrole}
	
	While we have seen that computable simulations explicate the notion of computability adequately, the question remains why admissible numberings should be computable in the first place.
	
	Recall the philosophical interpretation of Gödel's Second Theorem, according to which $T$ does not prove its consistency. This interpretation's justification rests on the hypothesis that we can formulate the question whether or not $T$ proves a sentence which expresses $T$'s consistency in a meaningful way. Indeed, the employed Gödel numbering enables us to interpret certain arithmetical closed terms and formul{\ae} as (names of) syntactic objects and properties of syntactical objects respectively. Let us call this the \textit{semantic role} of Gödel numberings. It allows us to not only reason in $T$ about numbers but also about expressions via their Gödel codes.
	
	The prevalent philosophical reading of Gödel's theorem as a limitative result rests on yet another crucial feature of numberings. What we are essentially concerned with, is the question of whether or not the consistency sentence is derivable by \textit{using only the theory $T$'s resources}. It is therefore pivotal in this context that admissible numberings ensure that reasoning about expressions via their representations in $\mathcal{L}$ requires no resources which lie outside the formal system~$T$. Let us call this the \textit{epistemic role} of Gödel numberings.
	
	As we have seen in Section~\ref{subsection:computablesimulations}, every numbering induces a pair $(G, \oaste)$, where $G$ is a set of codes and $\oaste$ is a $\ast$-simulating operation on $G$, which represents the structure of expressions. That this representation does not employ resources exceeding $T$ can be understood in virtue of $T$ \enquote{recognising} or \enquote{knowing} or \enquote{verifying} certain facts and properties concerning $(G, \oaste)$. Since $T$ is a formal theory, it can be seen to \enquote{recognise}, \enquote{know} and \enquote{verify} facts and properties by means of proving (formul{\ae} which express) them.
	
	So far, none of these facts and properties have been specified. A minimal approach is to require that $T$ \enquote{recognises} the two constituents of the given representation, namely the set $G$ and the operation $\oaste$. Accordingly, $T$ \enquote{knows} whether or not ${n \in G}$, and whether or not $l \oaste m=n$, for numbers $l$, $m$ and $n$. That is, $T$ binumerates both $G$ and $\oaste$ (i.e., its graph). But this is tantamount to $G$ being decidable and $\oaste$ being recursive.\footnote{See for instance \citep[Corollary II.7]{TarskiMostowskiRobinson}.} Hence by Definition~\ref{defi:compnumbering}, the underlying numbering is a computable simulation. 
	
	Note that even though the admissibility of a numbering is analysed in terms of $T$'s resources, the extracted notion of computable simulation is theory-independent. This rests on the fact that theories extending $\mathsf{R}$ binumerate exactly the same sets and functions, namely decidable and recursive ones respectively.
	
	However, one might require $T$ to recognise further properties regarding $(G, \oaste)$. For instance, $T$ might reasonably be required to verify the functionality of $\oaste$, i.e., to prove that $\oaste$ is total (cf.~Section~\ref{subsection:intensionality}). Since different theories extending $\mathsf{R}$ prove the totality of different classes of functions in general, this requirement yields a notion of admissibility which is sensitive to~$T$.
	
	As we will show in Section 4, even the minimal approach places enough restriction on admissible numberings to establish the invariance of Gödel's Second Theorem (see Theorem~\ref{maintheorem:invariance}). It is therefore sufficient for the purpose of this paper to conclude that admissible numberings are computable simulations.\footnote{ As a result of the above analysis, the computability of a numbering can be seen to provide a constraint which bars resources beyond $T$ to enter into reasoning about expressions. This is reminiscent of a popular conception of computation, according to which computational processes \enquote{have access only to the formal properties of \ldots representations} \citep[p. 231]{Fodor1981}, not to their content \citep[p. 254]{Egan2010}. The the involved notions' lack of clarity notwithstanding, this conception may serve as an informal yet insightful heuristic to see why the employed numberings are inadmissible in the deviant cases above. For instance, which $\eta$-code is assigned to an expression $\chi$ is determined by whether or not $\mathbb{N} \models \chi$. The fact that the definition of $\eta$ has access to this semantic property of expressions, namely their truth values, thus renders $\eta$ inadmissible. Similarly, both the \mbox{$\delta$-code} and the $\zeta$-code of an expression $\chi$ is defined by recourse to whether or not $T \vdash \chi$. Since the set of non-theorems has (classically, i.e., relative to standard numberings) complexity $\Pi^0_1$ this property once again exceeds the resources of $T$. Thus in all these cases, the employed numberings violate the \enquote{formality condition} imposed by computability, as they significantly resort to non-formal properties of strings, exceeding $T$'s resources. Appealing to these properties in the numberings' definitions causes sets of expressions which are classically non-decidable or non-arithmetical to be coded by sets which are decidable relative to the respective deviant numbering. This discrepancy lies at the root of the resulting deviancy.}\\
	
	In the context of the semantic version of Tarski's theorem, we are concerned with the question of whether the set of true sentences is definable in the language $\mathcal{L}$, rather than with provability in a formal theory $T$. Accordingly, the epistemic role of numberings in this context is to ensure that reasoning about expressions does not resort to resources exceeding the language $\mathcal{L}$, rather than the theory $T$ as in the context of Gödel's theorem. This can be made precise by requiring that $\mathcal{L}$ \enquote{captures} or \enquote{expresses} relevant facts and properties about admissible representations $(G, \oaste)$ by means of defining them. Once again resorting to a minimal approach then yields the requirement that both $G$ and $\oaste$ (i.e., its graph) are definable in the standard interpretation $\mathbb{N}$ of $\mathcal{L}$. We call any numbering which gives rise to such a representation $(G, \oaste)$ an \textit{arithmetical simulation \textnormal{(}of $A^\ast$\textnormal{)}}.
	
	According to the proposed analysis, the notion of admissibility is thus sensitive to the considered metamathematical context. While admissible numberings are computable simulations in the context of Gödel's Second Theorem and other limitative results which primarily pertain to provability-in-a-theory (such as the syntactic version of Tarski's Theorem), they are required to be arithmetical simulations in the context of the semantic version of Tarski's Theorem.
	
	In the remainder of this paper I will consider  the admissibility of numberings exclusively in the context of Gödel's Second Theorem (with Theorem~\ref{semanticTarski} and footnote~\ref{ftn:semanticTarski} as exceptions):
	
	\begin{description}
		\item[Computable Simulativity]
		Every admissible numbering is a computable simulation.
	\end{description}
	
	I now turn briefly to other proposals found in the literature which tie computational considerations to the admissibility of numberings.
	
	\subsection{Other Routes to Admissibility}
	\label{subsection:othercriteria}
	
	Let $\gamma$ denote the standard numbering introduced in \citep[Section 15.1]{Smith}.\footnote{Our alphabet $A$ deviates from the one considered in \citep{Smith}. In the remainder of this paper I assume that $\gamma$ is slightly adapted in order to code strings over $A$.} According to \cite{Smith} \enquote{the key feature of our Gödelian scheme [$\gamma$] is this: there is a pair of algorithms, one of which takes us from an expression to its code number, the other of which takes us back again from the code number to the original expression} (p.~126). Given any other comparable numbering $\alpha$, converting the $\alpha$-code of a certain expression~$\phi$ into its $\gamma$-code involves running through two algorithms (first mapping $\alpha(\phi)$ to $\phi$ and then mapping $\phi$ to $\gamma(\phi)$), which by assumption do not involve any open-ended searches (ibid.). Conversely, $\gamma$-codes can be converted into $\alpha$-codes in a similar fashion. Hence, \enquote{there is a p.r.\ function $tr$ which \enquote{translates} code numbers according to [$\alpha$] into code numbers under our official Gödelian scheme [$\gamma$], and another p.r.\ function $tr^{-1}$ which converts code numbers under our scheme back into code numbers under scheme [$\alpha$]} (ibid.).\footnote{The restriction to p.r.~functions for $tr$ and $tr^{-1}$ is justified by construing the aforementioned algorithms without open-ended \enquote{do while} loops (see \citep[Proposition I.5.8]{Odifreddi1}). However, I believe that there is no clear reason to impose this restriction on the algorithms, thus allowing $tr$ and $tr^{-1}$ to be recursive functions. This however does not bear on the subsequent discussion.} Let us call any numbering $\alpha$ for which such a pair of p.r.\ translation functions exist, \textit{computably translatable} to $\gamma$. The following criterion of admissibility is thus extracted by \cite{Smith}:
	
	\begin{description}
		\item[Computable Translatability]
		\textit{A numbering is admissible iff it is computably translatable to the standard numbering $\gamma$.}
	\end{description}
	
	While I agree with Smith's initial analysis that the key feature of admissibility consists essentially in intuitive computability, the proposed criterion fails to capture this feature in a conceptually satisfactory way. This deficiency can be seen to result from general difficulties in giving a precise account of the computability of numberings (see the remark following Definition~\ref{defi:compnumbering}). To circumvent these difficulties, Smith requires admissible numberings to be computably translatable to some designated standard numbering $\gamma$. The conceptual adequacy of the resulting criterion thus crucially rests on the admissibility of $\gamma$. For instance, employing one of the deviant numberings of Section~\ref{subsection:deviant} as a reference numbering instead of $\gamma$, would clearly render the criterion inadequate. What justifies the choice of $\gamma$ as the reference numbering is exactly its intuitive computability. Hence, instead of taking the conceptually prior feature of computability as characterising admissibility, a conceptual surrogate is employed.\footnote{This argument is borrowed from \cite[p.~268]{Rescorla2007} who discusses the admissibility of domain representations in the context of extending the Turing-Thesis to computational models over other domains.} Furthermore, Smith's justification for the existence of p.r.~translation functions between numberings which share the same \enquote{key feature} (i.e., come with a pair of encoding and decoding algorithms) crucially relies on informal reasoning about computability. A more satisfactory approach would be to first formally characterise the conceptually prior notion of intuitive computability of numberings (cf.~Definition~\ref{defi:compnumbering}) and then mathematically \textit{prove} the existence of respective translation functions between them (cf.~Theorem~\ref{theorem:numberinginvariance} and Corollary~\ref{cor:compsimequivtosequencadmiss} below), employing the Church-Turing Thesis only in its \enquote{interpretive use} \cite[p. 275]{Smith}.
	
	A second approach can be found in \citep{Smullyan1961}, who bases the concept of admissibility of numberings on representability in \textit{elementary formal systems} (for definitions see \citep[p.~6]{Smullyan1961}). These systems can be seen as devices intended to \enquote{explicate the notion of \enquote{definability by recursion}} \citep[p.~2]{Smullyan1961} directly operating on strings, without the usual recourse to numberings. A set $S$ of strings which is definable by recursion in some elementary formal system is called formally representable. If in addition to $S$ also the complement of $S$ is formally representable, $S$ is called solvable. The admissible numberings of $S$ are then taken to be exactly those which preserve these recursion theoretic properties. More precisely, a numbering $\alpha \colon S \to \mathbb{N}$ is called \textit{EFS-admissible}, if for every subset $U \subseteq S$ the following holds:
	$U$ is formally representable iff $\alpha(U)$ is recursively enumerable and $U$ is solvable iff $\alpha(U)$ is decidable. This yields:
	
	\begin{description}
		\item[EFS-Admissibility]
		\textit{A numbering is admissible iff it is EFS-admissible.}
	\end{description}
	
	Once again this definition has a computational motivation, since the concept of formal representability is intended to capture the concept of computability, in the sense that formally representable sets characterise sets which are generated by a \enquote{computing machine} \citep[p.~9]{Smullyan1961}.
	
	An account of admissibility which \textit{prima facie} \enquote{strictly contains} the notion of computable simulativity is introduced in \citep{Manin}. A numbering $\alpha$ is called \textit{sequence-admissible}, if the image of $\alpha$ is decidable and the length function, the projection function given by $\langle i , \langle a_1, \ldots, a_n \rangle \rangle \mapsto a_i$ as well as the concatenation function are recursive.\footnote{More precisely\label{ftn:relrecursive}, the functions are recursive relative to $\langle \alpha,\operatorname{id} \rangle$,  $\langle \operatorname{id} \times \alpha, \alpha \rangle$ and $\langle \alpha \times \alpha, \alpha \rangle$ respectively.} This yields the following criterion:
	
	\begin{description}
		\item[Sequence-Admissibility]
		\textit{A numbering is admissible iff it is sequence-admissible.}
	\end{description}
	
	Despite the different conceptions of computability underlying these proposals, all three notions will be shown in Section 4 to be extensionally equivalent to computable simulativity (see Corollary~\ref{cor:compsimequivtosequencadmiss}). This endows the criterion of computable simulativity with a desirable amount of robustness.
	
	\subsection{Compositionality and Monotonicity}
	\label{subsection:otherroutes}
	
	I conclude this section by discussing other approaches to the admissibility of numberings based on compositionality and monotonicity rather than computational considerations. As we have seen in Section~\ref{subsection:epistemicrole}, numberings have a semantic role as they represent strings by natural numbers. It thus seems reasonable to require admissible numberings to satisfy some principle of compositionality. Accordingly, the semantic value of every complex expression can be taken to be determined by its structure and the semantic values of its subexpressions.\footnote{The following discussion does not depend on the specific conception of semantic value.} Since for a given numbering $\alpha$ the semantic value of a string is its $\alpha$-code, the compositionality of $\alpha$ can be understood as $\alpha \colon A^\ast \to G$ being an \mbox{$\Omega$-homomorphism}, where the signature $\Omega$ contains a binary (concatenation) function symbol $\concat$. But as we have already seen in Section~\ref{subsection:computablesimulations}, every numbering~$\alpha$ induces such a homomorphism. Hence every numbering is compositional in this sense.
	
	Even resorting to the \enquote{mode of presentation} does not yield any proper restriction on numberings. It is a common practice to first assign numbers to the alphabetical symbols and then to define the Gödel number of a string as a function of the Gödel numbers of its entries (e.g.~based on prime factorisation in the tradition of \citep{Godel1931} or on $k$-adic notation following \citep{Smullyan1961, Smullyan92}). In fact, any numbering $\alpha$ of $A^\ast$ can be reconstructed in this standard \enquote{bottom-up} fashion: firstly, map each symbol $s$ of the alphabet $A$ to its $\alpha$-code. Secondly, consider the \mbox{$\ast$-simulating} partial function $\oaste \colon \mathbb{N}^2 \to \mathbb{N}$, given in Section~\ref{subsection:computablesimulations}. These two steps then resemble the usual \enquote{bottom-up} definition of numberings and uniquely determine our initially given numbering, since $\alpha$ is the unique function which extends the assignment given in step 1 and is compositional with regard to the partial function given in step 2, i.e., is an $\Omega$-homomorphism.\footnote{Algebraically speaking, $(A^\ast,\ast)$ has the universal mapping property for the class of $\Omega$-semigroups over $A$, see \citep[§10 \& §11]{Burris1981}.} Every numbering irrespective of its presentation can thus be reconstructed in the above manner, which is the prevalent mode of presentation for standard Gödel numbers found in the literature.
	
	In addition to the \enquote{bottom-up} fashion of defining numberings, it is ubiquitous in the literature to employ a \textit{monotonic} partial operation $\oaste$ in step 2. This gives rise to the following property of Gödel numberings:
	
	\begin{definition}
		A numbering $\alpha$ of $A^\ast$ is called monotonic, if $\alpha(s) \leq \alpha(t)$ for any $s,t \in A^\ast$ such that $s$ is a subexpression \textnormal{(}i.e., a substring\textnormal{)} of $t$.
	\end{definition}
	
	Monotonicity is a prevalent property shared by all standard numberings found in the literature (which are known to the author). One reason for this custom is that monotonicity ensures that the above bottom-up mode of presentation yields an \textit{injective} function and hence a numbering. Monotonicity thus serves as a technical constraint warranting the intended output of a particular construction principle (see for instance \citep[p. 220]{HilbertBernays1939}). Moreover, monotonicity provides a low arithmetical complexity of formul{\ae} representing syntactic properties, allowing the use of convenient tools (such as $\Sigma^0_1$-completeness) when proving certain metamathematical theorems.
	
	Apart from these technical considerations, on which conceptual grounds can monotonicity be justified? Can monotonicity be extracted as a necessary condition by analysing the admissibility of numberings as we did in the case of computability?
	
	As we have seen in Section~\ref{subsection:epistemicrole}, a crucial role of Gödel numberings is to allow reasoning about expressions in an arithmetical theory $T$. Admissible numberings were thus required to simulate certain structural properties of strings. As we have seen for instance, the operation $\oaste$ induced by $\alpha$ simulates the concatenation operation~$\ast$ on $A^\ast$ in virtue of the commutativity of diagram~\ref{commutdiagram} (in~\ref{subsection:computablesimulations}). The subexpression-relation~${\preceq}$ on $A^\ast$ can be reasonably taken as yet another important structural feature of strings. Hence for any admissible numbering, the induced semigroup $(G,\oaste)$ of Gödel numbers should simulate ${\preceq}$ appropriately. This can be made precise by extending the signature $\Omega$ by a binary relation symbol ${\sqsubseteq}$, yielding ${\Omega^+ := \{ \concat, {\sqsubseteq} \}}$, and by construing $(A^\ast,\ast,\preceq )$ as an $\Omega^+$-structure. A numbering $\alpha$ may then be called \textit{$\preceq$-preserving} if $\alpha \colon A^\ast \to G$ is an \mbox{$\Omega^+$-isomorphism}. It can be easily seen that in a similar way as in Section~\ref{subsection:computablesimulations}, for every~$\alpha$ a \enquote{\mbox{$\preceq$-simulation}} $\opreceq$ on $G$ can be defined, such that $(G,\oaste,\opreceq)$ is an \mbox{$\Omega^+$-structure} and $\alpha$ is an $\Omega^+$-isomorphism. Hence, once again, $\preceq$-preservation does not impose any restriction on numberings.
	
	By additionally drawing on the epistemic role of numberings, the induced structure $(G,\oaste,\opreceq)$ can be required to not exceed $T$'s resources, analogously to the analysis in Section~\ref{subsection:epistemicrole}. Accordingly, it can be required that $T$ \enquote{recognises}, and thus binumerates, the relation $\opreceq$, which is tantamount to $\opreceq$ being decidable. Yet, as in the proof of Corollary~\ref{cor:compsimequivtosequencadmiss} we can show that for every computable simulation~$\alpha$ also the \enquote{$\preceq$-simulating relation} $\opreceq$ is decidable. This condition therefore does not pose any additional constraint given the requirement of computable simulativity.
	
	As in the case of the $\ast$-simulating operation $\oaste$, one may require $T$ to verify further principles regarding $\preceq$. For instance, $T$ might be required to prove that any two strings $s,t$ are subexpressions of $s \ast t$. This results in a more restrictive notion of admissibility which is sensitive to the considered theory $T$. Since these principles are formulated in purely string-theoretical terms without access to the arithmetical properties of their codes, they result in limitations, for instance, regarding the growth rate of admissible numberings (cf.~Section~\ref{subsection:intensionality}). However, the constraint of monotonicity cannot be obtained using this approach.
	
	As suggested by an anonymous referee, another route to obtain monotonicity may proceed by conceiving of $\leq$ and $\preceq$ as parthood relations on numbers and strings respectively. Taking parthood as an important structural feature of strings, admissible numberings may then be required to preserve structure in virtue of representing the parthood relation $\preceq$ on strings by the parthood relation $\leq$ on numbers. I however do not think that this approach is satisfactory for our purposes, for the following reasons.
	
	Firstly, assuming that both relations ${\preceq}$ and ${\leq}$ are indeed similar in structure by being parthood relations, they still significantly differ with regard to another important structural aspect. While ${\leq}$ is a total order relation, there are strings which are incomparable with respect to~$\preceq$. For this reason, there exists no numbering such that ${\preceq}$ is represented by ${\leq}$, i.e., that ${\opreceq} = {\leq}$, clearly rendering the above requirement absurd. To require that ${\preceq}$ is represented by ${\leq}$ in virtue of ${\opreceq}$ being contained in ${\leq}$, i.e., that the underlying numbering is monotonic, appears to be an ad-hoc reaction to this structural discrepancy rather than based on conceptual grounds.
	
	Secondly, even if the adequacy of this last step is granted, further problems arise by this conception of structure preservation. This can be illustrated by basing the structure of strings on so-called successor functions $\mathsf{S}_a$, which append the alphabetical symbol $a \in A$ to any input string (see \citep{Corcoran1974}). Instead of employing the concatenation operation $\ast$, the set of strings can be generated from the empty string by using successor functions $\mathsf{S}_a$ for each $a \in A$. According to the above approach, admissible numberings can then be required to preserve this structure in virtue of representing the empty string and the successor functions~$\mathsf{S}_a$ by their arithmetical counterparts, namely, by $0$ and the successor operation ${n \mapsto n+1}$. However, even the requirement that (the graph of) the tracking function of \textit{some} $\mathsf{S}_a$ is contained in (the graph of) the arithmetical successor function cannot be satisfied by any Gödel numbering. As in the case of monotonicity, this requirement stems from a loose and almost metaphorical conception of structure similarity and preservation. In some sense, the empty string and the successor functions $\mathsf{S}_a$ indeed have the structure of $0$ and the arithmetical successor operation, namely in virtue of being generators and constructors respectively. However, they once again differ crucially in other structural respects, resulting in the absurdity of the extracted requirement.
	
	Thirdly, while I agree that admissible numberings should preserve the structure of the subexpression relation $\preceq$, it is not clear why structure preservation should entail constraints on the arithmetical content of $\opreceq$, such as the containment of $\opreceq$ in $\leq$, in the first place. Rather, I take an admissible numbering $\alpha$ to preserve the structure of $\preceq$ in virtue of the decidability of whether or not a number $n$ $\alpha$-codes a subexpression with $\alpha$-code $m$ for any $m,n \in \mathbb{N}$, or that the arithmetisations of certain additional string-theoretic properties regarding $\preceq$ are provable in $T$, as extracted from my initial analysis given above. As we have already seen, monotonicity can not however be derived from this conception of structure preservation.
	
	Finally, the very assumption that both ${\preceq}$ and ${\leq}$ are parthood relations appears to rely on specific and quite arbitrary representational choices of strings and numbers. While $\preceq$ may be understood as a parthood relation on strings via their natural interpretation as geometric objects, it is less clear how $\leq$ can be given such a mereological interpretation. One way is to represent numbers as strings of strokes. Another approach is to represent numbers as Von Neumann ordinals, and to understand the parthood relation on sets as the containment relation $\subseteq$ (see e.g.~\citep{Lewis1991}). Then $\leq$ is indeed a parthood relation on numbers. However, this interpretation breaks down for other reasonable representations, such as Zermelo ordinals or equivalence classes of finite sets under the equivalence relation of equinumerosity. 
	Yet another approach is to construe strings and numbers as free algebras and to extract a parthood relation from their underlying term algebras (see \citep{Burris1981}). In the case of numbers, this approach indeed yields $\leq$ as our parthood relation. If we take strings to be generated by the concatenation operation, then $A^\ast$ can be viewed as a free semigroup and we extract $\preceq$ as a parthood relation on strings. However, as we have seen above, we can also take strings to be generated from the empty string by the successor functions $\mathsf{S}_a$. This results in a different free algebra, according to which the parts of any string are exactly its initial strings. Similarly, parthood holds between strings and their final strings if we conceive of $\mathsf{S}_a$ as prefixing, instead of appending, $a$ to input strings.
	
	In addition to these conceptual reservations, the obtained results also suggest that monotonicity is a property orthogonal to the invariance problems considered in this paper.\footnote{This is not to say that monotonicity cannot serve as a reasonable and well-motivated constraint in a different metamathematical context.} Namely, there exist both monotonic and non-monotonic numberings which satisfy Gödel's Second Theorem as well as Tarski's theorems (take any standard numbering and Visser's (\citeyear{Visser1989}) self-referential numbering $\hash$ respectively), as well as monotonic and non-monotonic numberings which violate these theorems (e.g., $\zeta$ and $\delta$ respectively).
	
	Consequently, in what follows, admissible numberings are not required to be monotonic. This also permits a more general approach, since monotonicity constitutes a proper constraint on computable simulations.\footnote{Visser's (\citeyear{Visser1989}) $\hash$ is an example of a non-monotonic computable simulation.}
	
	\section{Invariance of Gödel's Second Theorem}
	\label{section:proofofinvariance}
	
	Before we turn to the proofs of the desired invariance claims, it will be convenient to introduce an equivalence relation on numberings, which preserves important computational properties. By a powerful theorem due to \cite{Malcev1961}, all admissible numberings will be shown to be equivalent in this sense. From this we can in particular conclude that computable simulativity is extensionally equivalent to the criteria of computable translatability, sequence-admissibility and EFS-admissibility, introduced in Section~\ref{subsection:othercriteria}.
	
	\subsection{Equivalence of Numberings}
	\label{subsection:equivofnumberings}
	
	\begin{definition}[\cite{Manin}]
		\label{defi:equivalentnumberings}
		Let $S$ be a set. We call two numberings $\alpha$, $\beta$ of $S$ equivalent \textnormal{(}and write $\alpha \sim \beta$\textnormal{)}, if $\alpha \circ \beta^{-1} \colon \beta(S) \to \mathbb{N}$ and $\beta \circ \alpha^{-1} \colon \alpha(S) \to \mathbb{N}$ are recursive.
	\end{definition}
	
	It can be easily checked that $\sim$ is a partial equivalence relation, i.e., $\sim$ is symmetric and transitive. If we require numberings to have a decidable image, then $\sim$ is also reflexive. In particular, $\sim$ is an equivalence relation on the set of computable simulations of $S$.
	
	A subset $R \subseteq S$ is called decidable, recursively enumerable, arithmetical relative to a numbering $\alpha$, if the set $\alpha(R)$ has the respective property. The following lemma shows that these properties are invariant with regard to equivalent computable simulations.\footnote{This is in principle Lemma VII.1.5 in \citep{Manin}.}
	
	\begin{lem}
		\label{lemma:manin:invarianceofequivalentnumb}
		Let $S_i$ be sets and let $\alpha_i$ and $\beta_i$ be numberings of $S_i$ such that ${\alpha_i \sim \beta_i}$, for $i \in \{1, \ldots ,k\}$ and $k \in \mathbb{N}$. Then for any subset $R \subseteq S_1 \times \ldots \times S_k$ the set $\alpha_1 \times \ldots \times \alpha_k (R) := \{ \langle \alpha_1(s_1), \ldots, \alpha_k(s_k) \rangle \mid \langle s_1,\ldots,s_k \rangle \in R \}$ is decidable, recursively enumerable, arithmetical iff $\beta_1 \times \ldots \times \beta_k (R)$ has the respective property.
	\end{lem}
	
	The next theorem due to \cite{Malcev1961} serves as a powerful tool in proving invariance.\footnote{A proof of a more general version of this theorem can be found in \citep[Section 4]{Malcev1961}. This theorem crucially relies on the finiteness of our alphabet $A$.}
	
	\begin{thm}
		\label{theorem:numberinginvariance}
		Any two computable simulations $\alpha$ and $\beta$ of $A^\ast$ are equivalent, i.e., $\alpha \sim \beta$.
	\end{thm}
	
	From Lemma~\ref{lemma:manin:invarianceofequivalentnumb} and Theorem~\ref{theorem:numberinginvariance} we can conclude the extensional equivalence of computable simulativity and the three criteria considered in Section~\ref{subsection:othercriteria}.
	
	\begin{cor}
		\label{cor:compsimequivtosequencadmiss}
		Let $\alpha$ be a numbering of $A^\ast$. The following are equivalent:
		\begin{enumerate}[noitemsep]
			\item $\alpha$ is a computable simulation of $A^\ast$;
			\item $\alpha$ is computably translatable to the standard numbering $\gamma$;
			\item $\alpha$ is sequence-admissible;
			\item $\alpha$ is EFS-admissible.
		\end{enumerate}
	\end{cor}
	
	\begin{proof}
		$(3 \Rightarrow 1)$ and $(4 \Rightarrow 1)$ are immediate. $(1 \Leftrightarrow 2)$ and $(1 \Rightarrow 4)$ follow from Theorem~\ref{theorem:numberinginvariance} and Lemma~\ref{lemma:manin:invarianceofequivalentnumb}.
		
		$(1 \Rightarrow 3)$: It is easy to check that Smith's (\citeyear{Smith}) numbering $\gamma$ is a computable simulation as well as sequence-admissible. Then the length function and the concatenation function is recursive relative to $\langle \gamma,\operatorname{id} \rangle$ and $\langle \gamma \times \gamma, \gamma \rangle$ respectively (see footnote~\ref{ftn:relrecursive}). For any computable simulation $\alpha$ of $A^\ast$, these functions are also recursive relative to $\langle \alpha,\operatorname{id} \rangle$ and  $\langle \alpha \times \alpha, \alpha \rangle$ respectively, since $\alpha \sim \gamma$ (Theorem~\ref{theorem:numberinginvariance}) and r.e.\ sets are invariant with regard to equivalent numberings by Lemma~\ref{lemma:manin:invarianceofequivalentnumb}.
	\end{proof}
	
	\subsection{Invariance Proofs}
	
	We start by proving the invariance of the semantic version of Tarski's Theorem, which is also shown in \cite[p. 240]{Manin}:\footnote{\label{ftn:semanticTarski}In order to prove Theorem~\ref{semanticTarski} it is sufficient to require that $\alpha$ is an arithmetical simulation, as defined in Section~\ref{subsection:computablesimulations}. In fact, the proof of Theorem~\ref{theorem:numberinginvariance} can be adapted in order to show that all arithmetical simulations are \enquote{arithmetically equivalent}, i.e., the translations of two arithmetical simulations are arithmetical (see Definition~\ref{defi:equivalentnumberings}). Then a version of Lemma~\ref{lemma:manin:invarianceofequivalentnumb} can be proven, showing that arithmetical sets are preserved by arithmetically equivalent numberings. \cite[p.~23]{Smullyan92} contains a similar observation. This slightly stronger theorem establishes the invariance of the semantic version of Tarski's Theorem regarding numberings which are admissible in the context of this theorem (cf.~Section~\ref{subsection:computablesimulations}).}
	
	\begin{thm}
		\label{semanticTarski}
		For all computable simulations $\alpha$ \textnormal{(}of $A^\ast$\textnormal{)}, the set $\{ \alpha(\phi) \mid \mathbb{N} \models \phi  \}$ is not arithmetical.
	\end{thm}
	
	\begin{proof}
		Let $\alpha$ be any computable simulation of $A^\ast$. Since $\gamma$ is standard, Tarski's Theorem holds with regard to $\gamma$, i.e., $\{ \gamma(\phi) \mid \mathbb{N} \models \phi \}$ is not arithmetical. By Lemma~\ref{lemma:manin:invarianceofequivalentnumb} and Theorem~\ref{theorem:numberinginvariance} also $\{ \alpha(\phi) \mid \mathbb{N} \models \phi  \}$ is not arithmetical.
	\end{proof}
	
	In order to establish the invariance of Gödel's Second Theorem, we first formalise certain properties of the equivalence of numberings in $\mathsf{R}$.
	
	\begin{definition}	\label{defi:injrepr}
		Let $\alpha \colon B \to \mathbb{N}$ be a function with $B \subseteq \mathbb{N}$ and let $f(x,y) \in \mathit{Fml}$ be a binumeration of \textup(the graph of\textup) $\alpha$. We call $f(x,y)$ an \textit{inj-binumeration} of $\alpha$ if $\mathsf{R} \vdash \forall x, y\, ( f(x,\overline{m}) \wedge f(y,\overline{m}) \rightarrow x=y )$, for all $m \in \alpha(B)$.
	\end{definition}
	
	\begin{lem}	\label{lemma:existreprasinjective}
		Let $B \subseteq \mathbb{N}$. For each injective recursive function $\alpha \colon B \to \mathbb{N}$ there exists an inj-binumeration thereof.
	\end{lem}
	
	\begin{proof}
		Let $\alpha \colon B \to \mathbb{N}$ be an injective recursive function and let $\alpha'$ be a total recursive function extending $\alpha$ (see Section~\ref{subsection:prelimin}). By Theorem II.6 of \citep{TarskiMostowskiRobinson} there is a binumeration $g(x,y)$ of $\alpha'$ and a binumeration $\beta(x)$ of $B$. Then ${f'(x,y) \equiv \beta(x) \wedge g(x,y)}$ is a binumeration of $\alpha$. We define
		\[
		f(x,y) \equiv f'(x,y) \wedge  \forall z \leq x\, ( f'(z,y) \rightarrow x=z ).
		\]	
		In order to show that $f(x,y)$ numerates $\alpha$, let $n, m \in \mathbb{N}$ be such that $\alpha(n)=m$. Since $f'(x,y)$ binumerates $\alpha$, we have $ \mathsf{R} \vdash f'(\overline{n},\overline{m})$ and $\mathsf{R} \vdash \neg f'(\overline{k},\overline{m})$ for all $k < n$. Thus $\mathsf{R} \vdash f'(\overline{k},\overline{m}) \rightarrow \overline{n} = \overline{k}$ for all $k < n$, which yields $\mathsf{R} \vdash f'(\overline{k},\overline{m}) \rightarrow \overline{n} = \overline{k}$ for all $k \leq n$. Using $\mathsf{R4}$ yields ($\ast$)
		$\mathsf{R} \vdash \forall z \leq \overline{n}\, (f'(z,\overline{m}) \rightarrow \overline{n} = z)$. Thus $\mathsf{R} \vdash f(\overline{n},\overline{m})$. Since $\mathsf{R} \vdash \neg f'(\overline{n},\overline{m}) \rightarrow \neg f(\overline{n},\overline{m})$ and $\mathsf{R}$ is consistent, $f(x,y)$ also binumerates $\alpha$.
		
		Let $m \in \alpha(B)$. Let $n \in \mathbb{N}$ be such that ($\ast \ast$) $\mathsf{R} \vdash f(\overline{n},\overline{m})$. In order to show that $f(x,y)$ is an inj-binumeration it suffices to show that $\mathsf{R} \vdash \forall x ( f(x,\overline{m}) \rightarrow x = \overline{n})$, which we prove in the following derivation in $\mathsf{R}$:
		\begin{nbeweis}
			\bwz{f'(\overline{n},\overline{m})}{$(\ast \ast)$}\\
			\bwz{f(x,\overline{m}) \rightarrow f'(x,\overline{m}) \wedge \forall z \leq x\, ( f'(z,\overline{m}) \rightarrow x =z )}{Definition of $f(x,y)$}\\
			\bwz{\overline{n} \leq x \vee x \leq \overline{n}}{$\mathsf{R5}$}\\
			\bwz{\overline{n} \leq x \wedge \forall z \leq x\, ( f'(z,\overline{m}) \rightarrow x =z ) \rightarrow x = \overline{n}}{(1) and f.o.\ logic}\\
			\bwz{x \leq \overline{n} \wedge f'(x,\overline{m}) \wedge \forall z \leq \overline{n}\, (f'(z,\overline{m}) \rightarrow \overline{n} = z) \rightarrow x = \overline{n}}{f.o. logic}\\
			\bwz{\forall z \leq \overline{n}\, (f'(z,\overline{m}) \rightarrow \overline{n} = z)}{($\ast$)}\\
			\bwz{f'(x,\overline{m}) \wedge \forall z \leq x\, ( f'(z,\overline{m}) \rightarrow x =z ) \rightarrow x = \overline{n}}{(3)-(6)}\\
			\bwz{\forall x ( f(x,\overline{m}) \rightarrow x = \overline{n})}{(2),(7) and f.o.logic}
		\end{nbeweis}
	\end{proof}
	
	\begin{lem}
		\label{lemma:findenum:2}
		Let $B \subseteq \mathbb{N}$ and $\alpha \colon B \to \mathbb{N}$ be an injective recursive function. Let $f(x,y)$ be an inj-binumeration of $\alpha$ and let $\phi(x) \in \mathit{Fml}$. Then there exists ${\psi(y) \in \mathit{Fml}}$ such that $\mathsf{R} \vdash f(\overline{n},\overline{m}) \rightarrow (  \phi(\overline{n}) \leftrightarrow \psi(\overline{m}))$, for all $n \in \mathbb{N}$ and $m \in \alpha(B)$.
		
		If moreover $T$ is $\Sigma^0_k$-sound for some $k \geq 1$ such that $f(x,y) \in \Sigma^0_k$ and $\phi(x)$ is a $\Sigma^0_1$-numeration of $C \subseteq B$ in $T$, then $\psi(y)$ numerates $\alpha(C)$ in $T$.
	\end{lem}
	
	\begin{proof}
		Define $\psi(y) \equiv \exists z \left( f(z,y) \wedge \phi(z )\right)$. Using the fact that $f(x,y)$ is an inj-binumeration of $\alpha$, it is straightforward to show that ${\mathsf{R} \vdash f(\overline{n},\overline{m}) \rightarrow (  \phi(\overline{n}) \leftrightarrow \psi(\overline{m}))}$.
		
		Assume now that $T$ is $\Sigma^0_k$-sound for some $k \geq 1$ and $f(x,y) \in \Sigma^0_k$.
		Let $\phi(x)$ be a $\Sigma^0_1$-numeration of $C \subseteq B$ in $T$. In order to show that $\psi(y)$ numerates $\alpha(C)$ in $T$, let $m \in \mathbb{N}$ be such that $T \vdash \psi( \overline{m})$, i.e., $T \vdash \exists z \left(  f(z,\overline{m}) \wedge \phi(z )\right)$ by definition of $\psi$. Since  $\exists z \left(  f(z,\overline{m}) \wedge \phi(z )\right) \in \Sigma^0_k$, there exists $n \in \mathbb{N}$ such that $\mathbb{N} \models f(\overline{n},\overline{m}) \wedge \phi(\overline{n} )$ by the $\Sigma^0_k$-soundness of $T$. Since $f(x,y)$ is a binumeration of $\alpha$ it also defines $\alpha$, hence $\alpha (n)=m$. Since $\phi$ is a $\Sigma^0_1$-numeration of $C$ in $T$ we get $n \in C$ by $\Sigma^0_1$-completeness of $T$. Thus $m = \alpha(n) \in \alpha(C)$ obtains. The proof of the other direction is immediate.
	\end{proof}
	
	In addition to proving the invariance of Gödel's Second Theorem, the above method can also be used to simplify proofs of metamathematical theorems. For instance, they provide a proof of the Diagonal Lemma which avoids the usual tedious process of arithmetisation (in particular of the numeral and the substitution functions).
	
	\begin{lem}
		\label{diaglemma}
		Let $T \supseteq \mathsf{R}$, let $\alpha$ be an admissible numbering and let $\phi(x) \in \mathit{Fml}$ with free variable $x$. Then there exists a sentence $\lambda$ such that $T \vdash \phi(\gn{\lambda}{\alpha}) \leftrightarrow \lambda$.
	\end{lem}
	
	\begin{proof}
		Let $\hash$ be the numbering introduced in \citep[p. 159f.]{Visser1989}. It can be easily checked that $\hash$ is a computable simulation. Thus $\alpha \sim \hash$ by Theorem~\ref{theorem:numberinginvariance}. By Lemma~\ref{lemma:findenum:2} there exists a formula $\psi(y)$ such that for all $n \in \mathbb{N}$ and ${m \in \hash(A^\ast)}$ we have $(\ast)$ $\mathsf{R} \vdash f(\overline{n},\overline{m}) \rightarrow (  \phi(\overline{n}) \leftrightarrow \psi(\overline{m})),$ where $f$ binumerates $\hash \circ \alpha^{-1}$. By definition of $\hash$, for any given formula $\psi(y)$ there exists $k \in \mathbb{N}$ such that ${\hash( \psi(\overline{k}) ) = k}$ \citep[p. 159]{Visser1989}. Hence $\gn{\psi(\overline{k})}{\hash} \equiv \overline{k}$ and therefore $\vdash \psi(\gn{\psi(\overline{k})}{\hash} ) \leftrightarrow \psi(\overline{k})$. Setting $\lambda \equiv \psi(\overline{k})$, $(\ast)$ yields ${\mathsf{R} \vdash \phi(\gn{\lambda}{\alpha}) \leftrightarrow \psi(\gn{\lambda}{\hash})}$. Thus $\mathsf{R} \vdash \phi(\gn{\lambda}{\alpha}) \leftrightarrow \lambda$.
	\end{proof}
	
	Note that neither the definition of $\hash$ nor the proofs of Theorem~\ref{theorem:numberinginvariance} and Lemma~\ref{lemma:findenum:2} employ any arithmetisation of the usual syntactic properties and operations.
	
	\begin{cor}[Syntactic version of Tarski's Theorem]
		\label{syntacticTarski}
		Let $T \supseteq \mathsf{R}$ and let $\alpha$ be an admissible numbering. Then there exists no formula $\tau(x) \in \mathit{Fml}$ with free variable $x$ such that $T \vdash \phi \leftrightarrow \tau( \gn{\phi}{\alpha} )$ for every sentence $\phi$.
	\end{cor}
	
	Using Lemma~\ref{diaglemma} and the usual schematic \enquote{modal reasoning} we can also prove the invariance of Löb's Theorem as well as its formalised version regarding admissible numberings. In particular, we can prove the desired
	
	\begin{thm}[Invariance of Gödel's Second Theorem]
		\label{maintheorem:invariance}
		$T \not\vdash \neg \mathsf{Pr}^{\alpha}(\gn{\bot}{\alpha})$, for all admissible numberings $\alpha$ and arithmetical formul{\ae} $\mathsf{Pr}^{\alpha}(x)$ satisfying \textnormal{Löb($T, \alpha$)}.
	\end{thm}
	
	The reader might observe that the formulation of Theorem~\ref{maintheorem:invariance} already presupposes the employed provability predicates to satisfy Löb's conditions. Since these conditions allow for the schematic \enquote{modal reasoning} sufficient to prove Gödel's Second Theorem (or more generally, Löb's Theorem), the problem of the invariance of Gödel's Second Theorem with regard to numberings was essentially reduced to establishing an invariant version of the Diagonal lemma (Lemma~\ref{diaglemma}).
	
	However, no information about the \textit{existence} of provability predicates satisfying Löb's conditions with regard to varying numberings is provided by Theorem~\ref{maintheorem:invariance}. As a result, the reader might wonder how much Theorem~\ref{maintheorem:invariance} improves on the classical result by Hilbert, Bernays and Löb (Theorem~\ref{standardGodel2}). For it could be the case that only the typically employed, standard numberings yield non-trivial\footnote{Note that also the formula $x=x$ satisfies Löb's conditions.} provability predicates satisfying Löb's conditions, rendering Theorem~\ref{maintheorem:invariance} a trivial extension of the standard Theorem~\ref{standardGodel2} (at least from an extensional perspective).
	
	The next theorem shows that in fact every admissible numbering $\alpha$ allows for the construction of a non-trivial provability predicate satisfying Löb's conditions relative to $\alpha$. This establishes Theorem~\ref{maintheorem:invariance} as a proper extension of the classical results based on a specific numbering.
	
	Let $\Omega_1 \equiv \forall x,y\, \exists z\, {\sf o}(x,y,z)$ where $\mathsf{o}(x,y,z)$ is a standard binumeration of (the graph of) the polynomially growing function $\omega_1(x,y) = x^{\log y}$, where $\log x = \max \{y \mid 2^y \leq x \}$. Let $\mathsf{Exp} \equiv \forall x,y\, \exists z\, {\sf e}(x,y,z)$, where ${\sf e}(x,y,z)$ is a standard binumeration of (the graph of) the exponentiation function given by $(x,y) \mapsto x^y$.\footnote{The importance of employing \textit{intensionally correct} binumerations can be illustrated by the following example. Set $\mathsf{o'}(x,y,z) \equiv \mathsf{o}(x,y,z) \wedge \forall p < x + y \, \neg \mathsf{Proof}_{\mathsf{ZF}}(p,\gn{\bot}{})$, where $\mathsf{o}(x,y,z)$ is a standard binumeration of $\omega_1$ and $\mathsf{Proof}_{\mathsf{ZF}}$ is a standard $\mathsf{ZF}$-proof predicate. Clearly, $\mathsf{o'}(x,y,z)$ binumerates $\omega_1$. However, the totality claim $\Omega_1$ based on $\mathsf{o'}(x,y,z)$ implies the $\mathsf{ZF}$-consistency sentence $\neg \mathsf{Pr}_{\mathsf{ZF}}(\gn{\bot}{})$.}
	
	\begin{thm}
		\label{theorem:last}
		For all admissible numberings $\alpha$ and $\Sigma^0_2$-sound theories $T$ extending $\mathsf{I \Delta_0 + \Omega_1}$ or $\mathsf{I \Delta_0 + Exp}$, there exists a formula $\mathsf{Pr}_{T}^{\alpha}(x)$ satisfying \textnormal{Löb($T, \alpha$)} which numerates $\{ \alpha(\phi) \mid T \vdash \phi  \}$ in $T$.
	\end{thm}
	
	Before proving the theorem we show a useful auxiliary lemma.
	
	\begin{lem}
		\label{lemma:derivabilitycond}
		Let $\alpha$ and $\beta$ be equivalent numberings. Then there exists a binumeration $f(x,y)$ of $\beta \circ \alpha^{-1}$ such that for each $\mathsf{Pr^{\alpha}}(x)$ satisfying \textnormal{Löb($T, \alpha$)} there exists $\mathsf{Pr^{\beta}}(x)$  satisfying \textnormal{Löb($T, \beta$)}, such that for all $n \in \mathbb{N}$ and $m \in \beta(A^\ast)$
		\[
		\mathsf{R} \vdash f(\overline{n},\overline{m}) \rightarrow ( \mathsf{Pr^{\alpha}}(\overline{n}) \leftrightarrow \mathsf{Pr^{\beta}}(\overline{m})  ).	
		\]
		If moreover $T$ is $\Sigma^0_2$-sound and $\mathsf{Pr^{\alpha}}(x)$ is a $\Sigma^0_1$-numeration of $\alpha(T^{\vdash} )$ in $T$, then $\mathsf{Pr^{\beta}}(x)$ numerates $\beta(T^{\vdash} ) $ in $T$.
	\end{lem}
	
	\begin{proof}	By Lemma~\ref{lemma:existreprasinjective} \& \ref{lemma:findenum:2} there exists a binumeration $f(x,y)$ of $\beta \circ \alpha^{-1}$ and a formula $\mathsf{Pr^{\beta}}(x)$ such that
		\begin{equation}
		\label{equ1:existenceproof}
		\mathsf{R} \vdash f(\overline{n},\overline{m}) \rightarrow ( \mathsf{Pr^{\alpha}}(\overline{n}) \leftrightarrow \mathsf{Pr^{\beta}}(\overline{m})  ),
		\end{equation}
		for all $n \in \mathbb{N}$ and $m \in \beta(A^\ast)$. We show now that $\mathsf{Pr^{\beta}}(x)$ satisfies \textnormal{Löb($T, \beta$)}:	
		\begin{enumerate}[label=\roman*.]
			\item If $T \vdash \phi$, then $T \vdash \mathsf{Pr^{\alpha}}(\gn{\phi}{\alpha})$, since $\mathsf{Pr^{\alpha}}(x)$ satisfies Löb1$^\alpha$. We have
			\[
			\mathsf{R} \vdash f(\gn{ \phi}{\alpha},\gn{ \phi}{\beta}),
			\]
			since $f(x,y)$ binumerates $\beta \circ \alpha^{-1}$. It then follows from \ref{equ1:existenceproof} that
			\[
			\mathsf{R} \vdash  \mathsf{Pr^{\alpha}}(\gn{ \phi}{\alpha}) \leftrightarrow \mathsf{Pr^{\beta}}(\gn{ \phi }{\beta}).
			\]
			Thus $T \vdash \mathsf{Pr^{\beta}}(\gn{\phi}{\beta})$.
			\item
			Let $\phi, \psi \in \mathit{Fml}$. Since $f(x,y)$ binumerates $\beta \circ \alpha^{-1}$ we have
			\[
			\mathsf{R} \vdash f(\gn{ \phi \rightarrow \psi}{\alpha},\gn{ \phi \rightarrow \psi}{\beta}).
			\]
			Using \ref{equ1:existenceproof} we therefore get
			\begin{equation}
			\label{equ2:existenceproof}
			\mathsf{R} \vdash  \mathsf{Pr^{\alpha}}(\gn{ \phi \rightarrow \psi}{\alpha}) \leftrightarrow \mathsf{Pr^{\beta}}(\gn{ \phi \rightarrow \psi}{\beta}).
			\end{equation}
			Similarly, we show
			\begin{equation}
			\label{equ3:existenceproof}
			\mathsf{R} \vdash  \mathsf{Pr^{\alpha}}(\gn{ \phi}{\alpha}) \leftrightarrow \mathsf{Pr^{\beta}}(\gn{ \phi}{\beta})
			\end{equation}
			and
			\begin{equation}
			\label{equ4:existenceproof}
			\mathsf{R} \vdash  \mathsf{Pr^{\alpha}}(\gn{ \psi}{\alpha}) \leftrightarrow \mathsf{Pr^{\beta}}(\gn{ \psi}{\beta}).
			\end{equation}
			Since $\mathsf{Pr^{\alpha}}(x)$ satisfies Löb2$^\alpha$ and using \ref{equ2:existenceproof}-\ref{equ4:existenceproof} we conclude that
			\[
			T \vdash \mathsf{Pr^{\beta}}(\gn{ \phi \rightarrow \psi}{\beta}) \rightarrow (\mathsf{Pr^{\beta}}(\gn{ \phi}{\beta}) \rightarrow \mathsf{Pr^{\beta}}(\gn{  \psi}{\beta})).
			\]
			\item
			Let $\phi \in \mathit{Fml}$. As in (ii) we show
			\begin{equation}
			\label{equ5:existenceproof}
			\mathsf{R} \vdash  \mathsf{Pr^{\alpha}}(\gn{ \mathsf{Pr^{\alpha}}(\gn{ \phi}{\alpha})}{\alpha}) \leftrightarrow \mathsf{Pr^{\beta}}(\gn{ \mathsf{Pr^{\alpha}}(\gn{ \phi}{\alpha})}{\beta})
			\end{equation}
			
			Since $\mathsf{Pr^{\alpha}}(x)$ satisfies Löb3$^\alpha$ and using \ref{equ3:existenceproof} \& \ref{equ5:existenceproof} we get
			\begin{equation}
			\label{equ6:existenceproof}
			T \vdash \mathsf{Pr^{\beta}}(\gn{ \phi}{\beta}) \rightarrow \mathsf{Pr^{\beta}}(\gn{ \mathsf{Pr^{\alpha}}(\gn{ \phi}{\alpha})}{\beta})
			\end{equation}
			Application of Löb1$^\beta$ to \ref{equ3:existenceproof} then yields
			\begin{equation*}
			T \vdash \mathsf{Pr^{\beta}}(\gn{\mathsf{Pr^{\alpha}}(\gn{ \phi}{\alpha}) \rightarrow \mathsf{Pr^{\beta}}(\gn{ \phi}{\beta})}{\beta}).
			\end{equation*}
			Applying Löb2$^\beta$ then yields
			\begin{equation*}
			T \vdash \mathsf{Pr^{\beta}}(\gn{\mathsf{Pr^{\alpha}}(\gn{ \phi}{\alpha})}{\beta})  \rightarrow \mathsf{Pr^{\beta}}(\gn{\mathsf{Pr^{\beta}}(\gn{ \phi}{\beta})}{\beta}).
			\end{equation*}
			In combination with \ref{equ6:existenceproof} we thus obtain
			\[
			T \vdash \mathsf{Pr^{\beta}}(\gn{ \phi}{\beta}) \rightarrow \mathsf{Pr^{\beta}}(\gn{ \mathsf{Pr^{\beta}}(\gn{ \phi}{\beta})}{\beta}).
			\]
		\end{enumerate}
		Assume now that $T$ is $\Sigma^0_2$-sound and that $\mathsf{Pr^{\alpha}}(x)$ is a $\Sigma^0_1$-numeration of $\alpha(T^{\vdash} )$ in~$T$. Since $\alpha(A^\ast)$ is decidable, inspection of the proof of Lemma~\ref{lemma:existreprasinjective} shows that we can choose $f(x,y) \in \Delta^0_2$. Then $\mathsf{Pr^{\beta}}(x)$ numerates $\beta(T^{\vdash} )$ in~$T$ by Lemma~\ref{lemma:findenum:2}.
	\end{proof}
	
	\begin{proof}[Proof of Theorem~\ref{theorem:last}]
		Let $\gamma$ be a standard numbering and let $\mathsf{Pr}^{\gamma}(x)$ a standard $\Sigma^0_1$-provability predicate satisfying \textnormal{Löb($T, \gamma$)} and numerating $\gamma(T^\vdash)$ in $T$ (see for instance \citep{Wilkie1987} if $T \supseteq \mathsf{I \Delta_0 + \Omega_1}$ or \citep[Chapter 7]{Rautenberg} if $T \supseteq \mathsf{I \Delta_0 + Exp}$). By Theorem~\ref{theorem:numberinginvariance}, $\alpha \sim \gamma$. Lemma~\ref{lemma:derivabilitycond} then yields a formula $\mathsf{Pr}^{\alpha}(x)$ satisfying \textnormal{Löb($T, \alpha$)}, numerating $\alpha(T^\vdash)$ in~$T$.
	\end{proof}
	
	The above lemma also provides a different proof of Theorem~\ref{maintheorem:invariance}.
	
	\begin{proof}[Alternative proof of Theorem~\ref{maintheorem:invariance}]
		Since $\gamma$ is a computable simulation we have $\alpha \sim \gamma$, by Theorem~\ref{theorem:numberinginvariance}. Applying Lemma~\ref{lemma:derivabilitycond} yields a binumeration $f(x,y)$ of $\gamma \circ \alpha^{-1}$ and a formula $\mathsf{Pr}^{\gamma}(x)$ satisfying \textnormal{Löb($T, \gamma$)}, such that
		\[
		\mathsf{R} \vdash \neg \mathsf{Pr}^{\alpha}(\gn{\bot}{\alpha}) \leftrightarrow \neg \mathsf{Pr}^{\gamma}(\gn{\bot}{\gamma}).
		\]
		Since $\gamma$ is standard, $T \not\vdash \neg \mathsf{Pr}^{\gamma}(\gn{\bot}{\gamma})$. Thus $T \not\vdash \neg \mathsf{Pr}^{\alpha}(\gn{\bot}{\alpha})$.
	\end{proof}
	
	\subsection{A Note on Intensionality}
	\label{subsection:intensionality}
	
	Despite its satisfaction of Löb's conditions as well as its numeration of $T$-theorems, I believe that even if $\mathsf{Pr}^{\gamma}(x)$ is constructed canonically, the provability predicate $\mathsf{Pr}^{\alpha}(x)$ obtained in Theorem~\ref{theorem:last} hardly qualifies as being \textit{intensionally correct}. According to the proof of Lemma~\ref{lemma:findenum:2}, $\mathsf{Pr}^{\alpha}(x)$ is of the form $\exists z \left( f(z,x) \wedge \mathsf{Pr}^{\gamma}(z) \right)$, with $f$ being a binumeration of ${\alpha \circ \gamma^{-1}}$. The intensional correctness of $\mathsf{Pr}^{\alpha}(x)$ therefore presumably fails, if made precise by a version of the resemblance criterion (see \citep[p.~676]{HalbachVisser1} and footnote~\ref{ftn:meaningpost}). If intensional correctness is understood by recourse to meaning postulates (see \citep[p.~676]{HalbachVisser1}), it should be noted that computable equivalence is too weak to preserve proof-theoretic properties which are not merely schematic.
	
	For instance, any intensionally correct provability predicate might be taken to satisfy global versions of Löb's conditions (see for instance \textbf{D2$^\textup{G}$} and \textbf{D3$^\textup{G}$} in \cite[Section 2.3]{Kurahashi2019}), as opposed to the local (i.e., schematic) conditions employed in this paper. While the invariance of Gödel's Second Theorem with regard to these conditions remains unaffected, it can be shown that there is a computable simulation $\alpha$ such that the provability predicate $\mathsf{Pr}^{\alpha}(x)$ obtained in Theorem~\ref{theorem:last} does not satisfy these global conditions.
	
	Furthermore, computability is in general too weak to enforce an intensionally correct arithmetisation of even basic syntactic operations. Before I illustrate this point for $T = \mathsf{PA}$, the reader is reminded that $A^\ast$ denotes the set of strings over the alphabet $A$ of our language $\mathcal{L}$. The operation $\underline{\neg}$ maps each well-formed formula $\phi \in A^\ast$ to its negation $(\neg \phi)$. We may reasonably require that every intensionally correct arithmetisation of syntax is based on a numbering such that the tracking functions of $\underline{\neg}$ and the other constructor operations introduced in Section~\ref{subsection:prelimin} are \textit{provably total} in $\mathsf{PA}$ (see also \citep[p.~33]{Halbach2014}).
	
	\begin{definition}
		Let $T$ be an $\mathcal{L}$-theory and let $f \colon \mathbb{N}^k \to \mathbb{N}$ be a function. We say that $f$ is $\Sigma_n$-definable in $T$, if there exists a $\Sigma_n$-formula $\phi(x_1, \ldots, x_k,y)$ which defines the graph of $f$ such that $T \vdash \forall x_1, \ldots, x_k \exists ! y\, \phi(x_1, \ldots, x_k,y)$.
		
		The function $f$ is called provably total in $T$ if $f$ is $\Sigma_1$-definable in $T$.
	\end{definition}
	
	We now construct a computable simulation $\beta$ of the set of strings $A^\ast$, such that the tracking function of $\underline{\neg}$ is recursive but \textit{not} provably total in $\mathsf{PA}$. Suppose that the symbols of the alphabet $A$ are ordered. We now effectively enumerate $A^\ast$ using the length-first ordering, according to which $s < t$, if the length of $s$ is smaller than the length of $t$, or if $s$ and $t$ have the same length, $s$ is lexicographically smaller than $t$. Let $\psi_i$ be the $i$-th string in this enumeration. Let $h \colon \mathbb{N} \to \mathbb{N}$ be a strictly increasing and recursive function which is not provably total in $\mathsf{PA}$, such that $h(\mathbb{N})$ is decidable, $0 \notin h(\mathbb{N})$ and $| h(\mathbb{N}) | = | \mathbb{N} \setminus h(\mathbb{N}) | = | \mathbb{N} | $.\footnote{An example of such a function is $F_{\varepsilon_0}$, the stage $\varepsilon_0$ of the fast growing hierarchy. In fact, the hierarchy can be defined such that $F_{\varepsilon_0}(\mathbb{N})$ is $\Delta^0_0(\mathsf{exp})$-definable, see \citep{Sommer1995} and \citep{Freund2017}.}
	
	We define $\beta \colon A^\ast \to \mathbb{N}$ recursively as follows. Let $\beta(\psi_0) = 0$. Suppose now that $\beta(\psi_j)$ is defined for all $j < i$, then
	\begin{equation*}
	\beta(\psi_i) =
	\begin{cases}
	h(\beta(\psi_j))  & \text{if } \psi_i \equiv (\neg \psi_j), \\
	\min \{ n \mid \Forall j < i \ n \neq \beta(\psi_j) \Land n \notin h(\mathbb{N}) \} & \text{if } \Forall j \ \psi_i \not\equiv (\neg \psi_j).
	\end{cases}
	\end{equation*}
	
	The resulting numbering $\beta$ is bijective, and $\beta \sim \iota$, where $\iota \colon A^\ast \to \mathbb{N}$ is given by $\psi_i \mapsto i$. Since $\iota$ is a computable simulation (see \cite[chapter 9]{Smullyan1961} and Corollary~\ref{cor:compsimequivtosequencadmiss}), also $\beta$ computably simulates $A^\ast$, by Lemma~\ref{lemma:manin:invarianceofequivalentnumb}. However, the \mbox{$\beta$-tracking} function of $\underline{\neg}$ coincides with $h$ and is therefore not provably total in $\mathsf{PA}$. That is, there is no $\Sigma_1$-numeration $\mathsf{Neg}(x,y)$ of the graph of the function ${\beta(\phi) \mapsto \beta( (\neg \phi) )}$ such that $\mathsf{PA} \vdash \forall x\, \exists ! z\, {\sf Neg}(x,y)$. The definition of $\beta$ can be easily extended such that the $\beta$-tracking functions of all the operations $\underline{=}$, $\underline{\wedge}$, $\underline{\rightarrow}$, $\underline{\forall}$, etc.\ are recursive but not provably total in $\mathsf{PA}$. However, as already mentioned above, this might reasonably be taken as a minimal requirement for an intensionally correct arithmetisation when resorting to the meaning postulate approach.
	
	\section{Summary and Conclusions}
	\label{section:conclusion}
	
	The purpose of this paper has been to investigate the invariance of Gödel's Second Theorem regarding numberings. More precisely, I have examined the extent to which one can abstract away from the choice of the employed numbering in the formulation of Gödel's theorem which is based on Löb's conditions (see~Theorem~\ref{standardGodel2}).
	
	I have shown that \enquote{counterexamples} to several metamathematical theorems can be constructed, based on a na\"ive approach to numberings. For instance, there exists a numbering $\delta$ and a formula $\mathsf{Pr}^\delta(x)$ satisfying Löb's conditions (relative to~$\delta$) such that $\neg \mathsf{Pr}^\delta(\gn{\bot}{\delta})$ is a provable consistency sentence. Moreover, there is a numbering~$\eta$ and a formula $\mathsf{Tr}^\eta(x)$ defining the set of ($\eta$-codes of) true sentences, such that $\mathsf{PA} \vdash \mathsf{Tr}^\eta(\gn{\phi}{\eta}) \leftrightarrow \phi$, for all sentences $\phi$, violating both the semantic and syntactic version of Tarski's Theorem (see Section~\ref{subsection:deviant} and set $\mathsf{Tr}^\eta(x) := \mathsf{Pr}^\eta(x)$).
	
	These examples suggest that the invariance problem regarding numberings comes with some subtleties which are typically overlooked. I have argued that they do not however refute the usual philosophical interpretations of these theorems, since the employed numberings resort to resources exceeding the relevant theoretical framework, and are therefore not \textit{admissible} choices in the formalisation process. Since the relevant theoretical framework depends on the given metamathematical context, the admissibility of numberings has been analysed as a context-sensitive notion. In the context of Gödel's Second Theorem and the syntactic version of Tarski's Theorem for instance, the relevant theoretical framework consists of the considered formal theory~$T$. Accordingly, admissible numberings are required to allow for $T$ to verify, i.e., prove, certain basic properties of strings via their codes. Since the relevant theoretical framework in the context of the semantic version of Tarski's Theorem consists of the considered language $\mathcal{L}$ rather than a theory, admissible numberings in this context are required to allow for $\mathcal{L}$ to express, i.e., define, certain properties of strings via their codes. By resorting to a minimal set of properties of strings, every admissible numbering in the context of Gödel's Second Theorem has been shown to be a \textit{computable simulation}, while in the context of the semantic version of Tarski's Theorem every admissible numbering  is an \textit{arithmetical simulation} (see Section~\ref{subsection:epistemicrole}). It should be noted that admissible numberings are not required to be monotonic (a detailed discussion of monotonicity is contained in Section~\ref{subsection:otherroutes}).
	
	Finally, I have formulated and proven (meta-)mathematical invariance claims relative to these two precise necessary conditions of admissibility (see Section 4), namely the invariance of Gödel's Second Theorem and the syntactic version of Tarski's Theorem regarding computable simulations (see Theorem~\ref{theorem:numberinginvariance} and Corollary~\ref{syntacticTarski}) as well as the invariance of the semantic version of Tarski's Theorem regarding arithmetical simulations (see Theorem~\ref{semanticTarski} and footnote~\ref{ftn:semanticTarski}). Thus, once admissible numberings are singled out, invariance prevails and these theorems can be seen to be independent of the choice regarding their underlying numbering.
	
	I conclude by pointing towards potential future research avenues. It should be noted that the notion of admissibility used in this paper is too weak to ensure the invariance of metamathematical theorems which do not only resort to schematic quantification. For instance, the parametric version of the Diagonal Lemma (see \cite[Theorem III.2.1]{Hajek1998}) or the free-variable version of Löb's Theorem (see \citep[Theorem 4.1.7]{Smorynski1977}) appears to be outside the scope of this paper's methods. While Lemma~\ref{lemma:findenum:2} can be converted into a uniform version such that $\mathsf{R} + \text{\enquote{$\leq$ is a linear ordering}} \vdash \forall x, y \, (f(x,y) \rightarrow (  \phi(x) \leftrightarrow \psi(y)))$, the translation function binumerated by $f$ is not necessarily $T$-provably total. This is however a necessary ingredient for proving the invariance of these theorems along the lines of our proof of Lemma~\ref{diaglemma}. These remarks suggest that invariance results of such theorems require a more restrictive notion of admissibility, allowing further properties of the resulting translation functions to be verified within $T$. A philosophically adequate account of admissibility of numberings \textit{relative to a theory} is thus called for (cf.~Section~\ref{subsection:epistemicrole}).
	
	This observation is closely related to another direction of research into which the study of invariance may be extended. For instance, in addition to Gödel's Second Theorem itself being invariant, one might also require its proof and the process of arithmetisation to be independent of the employed numbering. This would rule out intensionally incorrect provability predicates such as those present in the proof of Theorem~\ref{theorem:last}, as well as non-standard proofs based on non-standard numberings (cf.~Lemma~\ref{diaglemma}). As noted in Section~\ref{subsection:intensionality}, an invariance claim along these lines requires translation functions of numberings to preserve stronger than merely schematic proof-theoretic properties. Since computable equivalence is too weak for this purpose, once again a more restrictive and theory-dependent notion of admissibility of numberings is needed. I believe that future work in this direction may further elucidate to what extent the choice of numberings bears on intensionality phenomena in metamathematics.
	
	\section{Appendix: Constructions of Deviant Numberings}
	\label{appendix}
	
	\paragraph{The Numbering $\bm{\delta}$.}
	For technical convenience we add the constant symbol $\mathsf{1}$ to~$\mathcal{L}$. The reader is reminded that the set of expressions is given by $A^\ast$ (with $\mathsf{1} \in A$). Let $T \supseteq \mathsf{PA}^-$ be any fixed consistent theory, where $\mathsf{PA}^-$ is the theory of the non-negative parts of discretely ordered rings (see \citep[p.\ 16f.]{Kaye}). We specify a numbering $\delta$ of $A^\ast$ such that the tracking functions of the operations $\underline{\mathsf{S}}$, $\underline{+}$, $\underline{\times}$, $\underline{\prime}$, $\underline{=}$, $\underline{\wedge}$ and $\underline{\forall}$ are recursive relative to $\delta$ and the set of ($\delta$-codes of) $T$-theorems is decidable. For technical convenience, we employ a derivability relation such that for every formula $\phi$ we have $T \vdash \phi$ iff $T \vdash \forall \phi$, where $\forall \phi$ denotes the universal closure of $\phi$.
	
	The basic idea is to construct $\delta$ such that all ($\delta$-codes of) $T$-theorems can be isolated from the remaining ($\delta$-codes of) strings of $A^\ast$ by the decidable property of parity. In order to do so, we first partition the even and odd numbers into infinite decidable sets, each corresponding to a certain syntactic category. Let ${\langle x,y \rangle \equiv \frac{(x+y+1)(x+y+2)}{2}+y}$ be the usual pairing function. We define the following functions:
	\begin{align*}
	& & \Theta^{\mathsf{junk}}(x,y) & = 12 \cdot \langle x,y \rangle +1,\\
	& & \Theta^{\mathsf{tm}}(x,y) & = 12 \cdot \langle x,y \rangle +3,\\
	\Lambda^{=}(x,y) & = 8 \cdot \langle x,y \rangle, & \Theta^{=}(x,y) & = 12 \cdot \langle x,y \rangle +5,\\
	\Lambda^{\neg}(x,y) & = 8 \cdot \langle x,y \rangle +2, & \Theta^{\neg}(x,y) & = 12 \cdot \langle x,y \rangle +7,\\
	\Lambda^{\wedge}(x,y) & = 8 \cdot \langle x,y \rangle + 4, & \Theta^{\wedge}(x,y) & = 12 \cdot \langle x,y \rangle +9,\\
	\Lambda^{\forall}(x,y) & = 8 \cdot \langle x,y \rangle +6, & \Theta^{\forall}(x,y) & = 12 \cdot \langle x,y \rangle +11.
	\end{align*}
	
	In what follows, we define $\delta$ such that all ($\delta$-codes of) terms are elements of $\Theta^{\mathsf{tm}} (\mathbb{N},\mathbb{N}) = \{ \Theta^{\mathsf{tm}} (m,n) \mid m,n \in \mathbb{N} \} = 12 \mathbb{N} + 3$, all ($\delta$-codes of) $T$-theorems of the form $(s = t)$ are elements of $\Lambda^{=} (\mathbb{N},\mathbb{N}) =  8 \mathbb{N}$, all ($\delta$-codes of) formul{\ae} of the form $(s = t)$ which are not $T$-theorems are elements of $\Theta^{=} (\mathbb{N},\mathbb{N})   = 12 \mathbb{N} +5$, all ($\delta$-codes of) of $T$-theorems of the form $(\phi \wedge \psi)$ are elements of $\Lambda^{\wedge} (\mathbb{N},\mathbb{N}) = 8 \mathbb{N} + 2$, etc. Moreover, all ($\delta$-codes of) junk-expressions, i.e., strings which are not well-formed expressions, are elements of $\Theta^{\mathsf{junk}} (\mathbb{N},\mathbb{N}) = 12 \mathbb{N} + 1$. Since
	\begin{align*}
	2 \mathbb{N} & = \Lambda^{=} (\mathbb{N},\mathbb{N}) \cup \Lambda^{\neg} (\mathbb{N},\mathbb{N}) \cup \Lambda^{\wedge} (\mathbb{N},\mathbb{N}) \cup \Lambda^{\forall} (\mathbb{N},\mathbb{N}) \text{ and}\\
	2 \mathbb{N}+1 & = \Theta^{\mathsf{junk}}(\mathbb{N},\mathbb{N}) \cup \Theta^{\mathsf{tm}}(\mathbb{N},\mathbb{N}) \cup \Theta^{=} (\mathbb{N},\mathbb{N}) \cup \Theta^{\neg} (\mathbb{N},\mathbb{N}) \cup \Theta^{\wedge} (\mathbb{N},\mathbb{N}) \cup \Theta^{\forall} (\mathbb{N},\mathbb{N})
	\end{align*}
	all $\delta$-codes of $T$-theorems are even, while the $\delta$-codes of all other expressions are odd. Thus $T \vdash \phi \text{ iff } \delta(\phi) \text{ is even}$, for every sentence $\phi$. Using $\Sigma^0_1$-completeness it is then easy to check that the $\Delta^0_0$-formula $\mathsf{Pr}(x) \equiv \exists y < x \ x = \overline{2} \times y$ satisfies \textnormal{Löb($T, \delta$)} and $T \vdash \neg \mathsf{Pr}(\gn{\bot}{\delta})$. However, $\mathsf{Pr}(x)$ will not define the set of $T$-theorems, since not every even number will code a $T$-theorem. We construct a $\Delta^0_0(\mathsf{exp})$-formula below which in addition to satisfying Löb's conditions also binumerates the set of $T$-theorems.
	
	In order to define the desired numbering $\delta$, we first construct numberings $\delta_i$ for every stage ${i < \omega}$. On the term level any standard numbering will do. In order to satisfy the aforementioned desired properties, we set $\delta_0(s) = \Theta^{\mathsf{tm}}(0, \gamma(s) )$, for each $s \in \mathit{Term}$, where $\gamma$ is Smith's (\citeyear{Smith}) standard numbering. The arithmetisation of the syntactic properties of and operations on terms proceeds as usual. In particular, the properties of being a ($\delta_0$-code of a) variable and being a ($\delta_0$-code of a) term are decidable. Moreover, the $\delta_0$-tracking functions of $\underline{\mathsf{S}}$, $\underline{+}$, $\underline{\times}$, $\underline{\prime}$ are recursive.
	
	In order to extend $\delta_0$ such that also $\underline{=}$ has a recursive tracking function, we first show that for any given expressions $s,t$, it is decidable whether or not $T \vdash s=t$ (with $s,t$ being possibly open terms).
	
	Since $T$ is $\Pi^0_1$-sound and proves the axioms of a commutative semiring, to each term $s(x_1,\ldots,x_l)$ there corresponds a unique polynomial $p_s \in \mathbb{N}[x_1,\ldots,x_l]$ such that $p_s = p_t$ iff $T \vdash s = t$. In order to show the decidability of $T \vdash s = t$, the rough idea is to define a rewriting system which effectively reduces each term $s(x_1,\ldots,x_l)$ to its unique normal form $p_s$. Since commutativity blocks the termination of such systems, we instead define a class-rewriting system which essentially operates on AC-congruence classes of terms, with $AC$ being the equational theory consisting of the associativity and commutativity axioms of both $+$ and $\cdot$. This gives rise to a system with the following properties:
	\begin{enumerate}[noitemsep]
		\item for each term there exists a unique normal form up to permutations under associativity and commutativity,
		\item each normal form $s'$ of $s$ is a term and $T \vdash s = s'$,
		\item normal forms of terms can be effectively computed,
		\item for any two $s,t$ and respective normal forms $s',t'$: $T \vdash s = t$ iff $s'  \sim_{AC} t'$.
	\end{enumerate}
	
	We first define the ordinary term rewriting system $R$, consisting of the following rules:
	\begin{enumerate}[noitemsep,label=\roman*.]
		\item $t \times (u + v) \rightarrow t \times u + t \times v$
		\item $\mathsf{S}t \rightarrow t + \mathsf{1}$
		\item $\mathsf{1} \times t \rightarrow t$
		\item $\mathsf{0} \times t \rightarrow \mathsf{0}$
		\item $t + \mathsf{0} \rightarrow t$
	\end{enumerate}
	
	In allowing the rewriting of a term by means of rewriting any $AC$-equivalent term, $R$ is thus extended to the class-rewriting system $R / AC$. In order to show that this system terminates, let $A = \{ n \geq 2 \mid n \in \mathbb{N} \}$ and consider the following polynomial weight functions:
	\begin{enumerate}[noitemsep,label=\alph*.]
		\item $\mathsf{0}_A = 2$
		\item $\mathsf{1}_A = 2$
		\item $\mathsf{S}_A = X + 4$
		\item $+_A = X + Y +1$
		\item $\times_A = XY$
	\end{enumerate}
	
	All polynomials (a)-(e) satisfy associativity and commutativity. Furthermore the polynomials $F_{l,r}$ (i.e., the result of subtracting the weight of the right hand side from the left hand side of a given rule) of the rules $\textnormal{(i)-(v)}$ are $X-1$, $1$, $X$, $2X-2$ and $3$ respectively, which are all strictly positive (over $A$). Thus $R / AC$ is terminating (see \cite[Chapter 6]{Terese}). Furthermore, $R$ is left-linear. The critical pairs of $R$ are of the form ${(u+v,\mathsf{1} \times u + \mathsf{1} \times v)}$, $(\mathsf{0},\mathsf{0} \times u + \mathsf{0} \times v)$ and $(t \times u, t \times u + t \times \mathsf{0})$. It can easily be checked that they converge, for instance, $\mathsf{1} \times u + \mathsf{1} \times v \rightarrow u + \mathsf{1} \times v \rightarrow u + v$. Hence the class-rewriting system $R / AC$ has the Church-Rosser property modulo $AC$ (see \cite[Chapter 7]{Dershowitz1991}).
	
	We may conclude that for each term, $R  / AC$ effectively computes a unique normal form up to permutations under associativity and commutativity. In order to decide whether or not $T \vdash s =t$ for two terms $s,t$, one proceeds as follows: first effectively rewrite $s$ and $t$ into their respective normal forms $s'$ and $t'$. Then check whether or not $s'  \sim_{AC} t'$. Since the word problem for $AC$ is decidable and $T \vdash s =t$ iff $s'  \sim_{AC} t'$, the decidability of $T \vdash s =t$ obtains. This decision process can be mimicked on the $\delta_0$-codes of terms in the usual way, yielding a corresponding decidable arithmetical property, which is binumerated by the formula $\mathsf{PrEqu}(x,y)$.\\
	
	In order to extend $\delta_0$ to a numbering of $A^\ast$ satisfying the desired properties, we exploit the fact that for any given expressions $\phi$, $\psi$ and $x$, whether or not $\phi \wedge \psi$ and $\forall x \phi$ are $T$-provable is already fully determined by the $T$-provability of $\phi$ and $\psi$ (and by the decidable fact of whether or not $x$ is a variable). Clearly this does not hold for $\neg$, as neither $T \vdash \neg \phi$ nor $T \not\vdash \neg \phi$ are entailed by $T \not\vdash \phi$. For this reason, $\underline{\neg}$ will not have a recursive $\delta$-tracking function.
	
	Let $\mathit{Junk}$ denote all strings of $A^\ast$ which are not well-formed expressions. Let $\{ \chi_n \mid n \in \mathbb{N} \}$ be an enumeration without repetitions of $T$-provable formul{\ae} which are of the form $(\neg \phi)$ and let $\{ \nu_n \mid n \in \mathbb{N} \}$ be an enumeration without repetitions of non-$T$-provable formul{\ae} of the form $(\neg \phi)$. Moreover, let $\{ \mu_n \mid n \in \mathbb{N} \}$ be an enumeration without repetitions of $\mathit{Junk}$.
	
	Set $\delta_0(\chi_n) = \Lambda^{\neg}(0,n) $, $\delta_0(\nu_n)= \Theta^{\neg}(0,n)$ and $\delta^{\mathsf{junk}}(\mu_n)= \Theta^{\mathsf{junk}}(0,n)$. We now successively extend the numbering $\delta_0$ of $\mathit{Term} \cup \{ (\neg \phi) \mid \phi \in \mathit{Fml} \}$ in $\omega$-many steps to a numbering of $A^\ast$, by defining a function $\delta_{i+1}$ for each $i < \omega$ as follows:
	{\small
		\begin{align*}
		\dom(\delta_{i+1}) & = 
		\dom(\delta_{i}) \cup \{ (s = t)  \mid s,t \in \dom(\delta_{i}) \cap \mathit{Term} \} \cup\\
		& \quad \cup \{ (\phi \wedge \psi), (\forall x \phi)  \mid \phi, \psi \in \dom(\delta_{i}) \cap \mathit{Fml} \Land x \in \dom(\delta_{i}) \cap \mathit{Var} \};\\
		\delta_{i+1}(\phi) & = \delta_i(\phi), \text{ if } \phi \in \dom(\delta_{i});\\
		\delta_{i+1}((s = t)) & =
		\begin{cases}
		\Lambda^{=}(\delta_i(s),\delta_i(t)), & \text{ if } s, t \in \dom(\delta_{i}) \cap \mathit{Term}, T \vdash (s = t),\\
		\Theta^{=}(\delta_i(s),\delta_i(t)), & \text{ if } s,t \in \dom(\delta_{i}) \cap \mathit{Term}, T \not\vdash (s = t);\\
		\end{cases}\\
		\delta_{i+1}((\phi \wedge \psi)) & =
		\begin{cases}
		\Lambda^{\wedge}(\delta_i(\phi),\delta_i(\psi)), & \text{ if } \phi, \psi \in \dom(\delta_{i}) \cap \mathit{Fml}, T \vdash \phi, \psi,\\
		\Theta^{\wedge}(\delta_i(\phi),\delta_i(\psi)), & \text{ if } \phi, \psi \in \dom(\delta_{i}) \cap \mathit{Fml}, T \not\vdash \phi \text{ or } T \not\vdash \psi;\\
		\end{cases}\\
		\delta_{i+1}((\forall x \phi)) & =
		\begin{cases}
		\Lambda^{\forall}(\delta_i(x),\delta_i(\phi)), & \text{ if } x \in \dom(\delta_{i}) \cap \mathit{Var}, \phi \in \dom(\delta_{i}) \cap \mathit{Fml},  T \vdash \psi,\\
		\Theta^{\forall}(\delta_i(x),\delta_i(\phi)), & \text{ if } x \in \dom(\delta_{i}) \cap \mathit{Var}, \phi \in \dom(\delta_{i}) \cap \mathit{Fml}, T \not\vdash \psi.
		\end{cases}
		\end{align*}
	}
	
	Finally, we set $\delta = \left( \bigcup_{i < \omega} \delta_i \right) \cup \delta^{\mathsf{junk}}$, with $\dom(\delta) = \left( \bigcup_{i < \omega} \dom(\delta_i) \right) \cup \mathit{Junk}$. By definition of the functions $\Theta^{\mathsf{junk}}$, $\Theta^{\mathsf{tm}}$, $\Lambda^{=}$, $\Theta^{=}$, etc., $\delta$ is an injective function and thus a numbering of $A^\ast$. In order to show that $\{ \delta(\phi) \mid T \vdash \phi \}$ can be binumerated by a $\Delta^0_0(\mathsf{exp})$-formula, we define
	\begin{align*}
	\mathsf{Pr}(x) \equiv & \exists s \ \mathsf{Seq}(s) \wedge x = (s)_{\mathsf{Lh}(s)} \wedge \forall i \leq \mathsf{Lh}(s) \big( \exists n \ (s)_i = \Lambda^{\neg}(0,n)\\
	& \vee \exists p,q \ \mathsf{Term}(p) \wedge \mathsf{Term}(q) \wedge \mathsf{PrEqu}(p,q) \wedge (s)_i = \Lambda^{=}(p,q)\\
	& \vee \exists j,k < i \ (s)_i = \Lambda^{\wedge}( (s)_j, (s)_k )\\
	& \vee \exists j < i, y \leq (s)_i \ \mathsf{Var}(y) \wedge (s)_i = \Lambda^{\forall}(y, (s)_j ) \big).
	\end{align*}
	Clearly, $\mathsf{Pr}(x)$ defines $\{ \delta(\phi) \mid T \vdash \phi \}$. As we have seen above, $\mathsf{PrEqu}(x,y)$ can be taken to be a $\Delta^0_0(\mathsf{exp})$-formula, as well as $\mathsf{Seq}(x), \mathsf{Term}(x)$ and $\mathsf{Var}(x)$. It remains to show that the variable $s$ can be bounded. To do so, construct a $\Delta^0_0(\mathsf{exp})$-definable height function $h(x)$ such that $h( \delta(\phi) ) = \# \{ \psi \mid \psi \text{ is a subformula of } \phi \}$. Then the length of the sequence (coded by) $s$ can be bounded by the number of subformul{\ae} of the theorem (coded by) $x$, i.e., by $h(x)$. Thus $\mathsf{Pr}(x)$ can be taken to be $\Delta^0_0(\mathsf{exp})$.
	
	In a similar way, it can be shown that $\{ \delta(\phi) \mid \phi \in \mathit{Fml} \}$ is $\Delta^0_0(\mathsf{exp})$-binumerable, by extending $\mathsf{Pr}(x)$ in each clause with the corresponding set $\Omega^\neg$, $\Omega^{=}$, $\Omega^{\wedge}$ or $\Omega^\forall$ of non-provable formul{\ae} respectively.
	
	We conclude by showing the operations $\underline{=}$, $\underline{\wedge}$ and $\underline{\forall}$ are recursive relative to $\delta$. As above, it can be shown that $\dom(\delta_i)$, for all $i < \omega$ as well as $\dom(\delta)$ is decidable. Moreover, the sets $\mathit{Var}$, $\mathit{Term}$, $\mathit{Fml}$ as well as $T^\vdash$ are decidable relative to $\delta$. 
	To then compute, for instance, the tracking function of $\underline{\wedge}$ given by $\oewedge \colon \delta(\mathit{Fml}) \times \delta(\mathit{Fml}) \to \mathbb{N}$, one first decides for a given input $(m,n)$ whether or not $\mathsf{Pr}(m)$ and $\mathsf{Pr}(n)$ are true. If both hold, then the output is $\Lambda^{\wedge}(m,n)$. If not, the output is $\Theta^{\wedge}(m,n)$. Hence, the tracking function $\underline{\wedge}$ is recursive. Showing the recursiveness of the remaining tracking functions proceeds in a similar way.
	
	\paragraph{Remark.}
	
	The deviant numbering $\delta$ is optimal in the following sense.
	
	\begin{lem}
		\label{lem:deltaisoptimal}
		There is no numbering $\delta'$ of $A^\ast$ such that the set of \textup($\delta'$-codes of\textup) \mbox{$T$-theorems} is decidable and \emph{all} the constructor operations introduced in Section~\ref{subsection:prelimin} are recursive relative to $\delta'$. 
	\end{lem}
	
	\begin{proof}[Proof sketch]
		Let $\delta'$ be a numbering of $A^\ast$ such that \emph{all} the constructor operations introduced in Section~\ref{subsection:prelimin} are recursive relative to $\delta'$. Let $\gamma$ be Smith's (\citeyear{Smith}) standard numbering. Clearly, all the constructor operations introduced in Section~\ref{subsection:prelimin} are also recursive relative to $\gamma$. By slightly generalising Mal'cev's theorem (see Theorem~\ref{theorem:numberinginvariance} and \citep[Section 4]{Malcev1961}), we can show that $\delta'$ and $\gamma$ are equivalent numberings of $\mathit{Term} \cup \mathit{Fml} \subset A^\ast$, i.e., of the well-formed $\mathcal{L}$-expressions. Hence, by Lemma~\ref{lemma:manin:invarianceofequivalentnumb}, the set of \textup($\delta'$-codes of\textup) $T$-theorems cannot be decidable.
	\end{proof} 
	
	In short: invariance is maintained as long as all constructor operations are required to be recursive. The construction of $\delta$ shows that deviance lurks as soon as we drop the requirement of recursiveness for $\underline{\neg}$. Similarly, deviant results can occur once the recursiveness requirement is given up for $\underline{\forall}$, as witnessed by the numbering $\eta$ constructed below. However, we cannot expect to find a deviant analogue to $\delta$ and $\eta$ if we relax the requirement of recursiveness for $\underline{\wedge}$.\footnote{I am grateful to an anonymous referee for suggesting this lemma to me.}
	
	\begin{lem}
		Let $T$ be $\Sigma_1$-sound. There is no numbering $\delta'$ of $A^\ast$ such that the set of \textup($\delta'$-codes of\textup) $T$-theorems is decidable and the constructor operations $\underline{\mathsf{S}}$, $\underline{+}$, $\underline{\times}$, $\underline{\prime}$, $\underline{=}$, $\underline{\neg}$ and $\underline{\forall}$ are recursive relative to $\delta'$.
	\end{lem}
	
	\begin{proof}[Proof sketch]
		Let $\delta'$ be a numbering of $A^\ast$ such that the set of \textup($\delta'$-codes of\textup) $T$-theorems is decidable and the operations $\underline{\mathsf{S}}$, $\underline{+}$, $\underline{\times}$, $\underline{\prime}$, $\underline{=}$, $\underline{\neg}$ and $\underline{\forall}$ are recursive relative to $\delta'$. Let $\mathit{Term}^- \cup \mathit{Fml}^- \subset A^\ast$ denote the conjunction-free fragement of~$\mathcal{L}$, i.e., the set of well-formed $\mathcal{L}$-expressions without conjunction. As in the proof of Lemma~\ref{lem:deltaisoptimal} we can show that Smith's $\gamma$ and $\delta'$ are equivalent numberings of $\mathit{Term}^- \cup \mathit{Fml}^-$. Hence, by Lemma~\ref{lemma:manin:invarianceofequivalentnumb}, $\{ \phi \in \mathit{Fml}^- \mid T \vdash \phi  \}$ is decidable relative to~$\gamma$. It is therefore decidable (in the classical sense) whether or not ${T \vdash \exists \vec{y} \, P(\vec{x},\vec{y}) = 0}$, for each polynomial $P(\vec{x},\vec{y}) = 0$ with integer coefficients, since $\exists \vec{y} \, P(\vec{x},\vec{y}) = 0 \in \mathit{Fml}^-$. Since $T$ is $\Sigma_1$-sound, we conclude that every Diophantine set is decidable. Hence, by the DPRM-theorem (see \citep{Davis1973}), every recursively enumerable set is decidable. But this is absurd.
	\end{proof}
	
	We can show similarly that there is no numbering $\delta'$ of $A^\ast$ such that the set of \textup($\delta'$-codes of\textup) true sentences is decidable and the operations $\underline{\mathsf{S}}$, $\underline{+}$, $\underline{\times}$, $\underline{\prime}$, $\underline{=}$, $\underline{\neg}$ and $\underline{\forall}$ are recursive relative to $\delta'$.

	\paragraph{The Numbering $\bm{\theta}$.}
	
	We now construct a numbering $\theta$ of $A^\ast$ such that the tracking function of $\underline{\neg}$ is recursive and the set of ($\theta$-codes of) $T$-theorems is decidable, for any fixed consistent theory $T \supseteq \mathsf{R}$. As opposed to the above construction, we here employ the standard derivability relation such that only closed sentences are \mbox{$T$-theorems}. We then use the fact that membership of $\neg \phi$ in the sets ${\{ \phi \mid T \vdash \phi \}}$, $\{ \phi \mid T \vdash \neg \phi \}$ and $\{ \phi \mid T \not\vdash \phi, T \not\vdash \neg \phi \}$ is already fully determined by the respective membership of $\phi$. This does however not hold for $\phi \wedge \psi$, since for instance ${T \vdash \neg (\phi \wedge \neg \phi)}$ but $T \not\vdash \neg (\phi \wedge \phi)$, for any $T$-independent $\phi$. Hence, in what follows, $\theta(\phi \wedge \psi)$ is not defined along the lines of $\theta(\neg \phi)$ and thus $\underline{\wedge}$ does not have a recursive $\theta$-tracking function.
	
	Let $\{ \chi_n \mid n \in \mathbb{N} \}$ be an enumeration without repetitions of $T$-theorems which are not of the form $(\neg \phi)$, let $\{ \mu_n \mid n \in \mathbb{N} \}$ be an enumeration without repetitions of $T$-refutable sentences which are not of the form $(\neg \phi)$ and let $\{ \nu_n \mid n \in \mathbb{N} \}$ be an enumeration without repetitions of expressions not of the form $(\neg \phi)$ which are neither $T$-provable nor $T$-refutable (including non-well-formed expressions).
	
	First, we set $\theta_0(\chi_n) = 3 \cdot \langle 0, n \rangle$, $\theta_0(\mu_n) = 3 \cdot \langle 0, n \rangle + 1$ and $\theta_0(\nu_n) = 3 \cdot \langle 0, n \rangle +2$, with $\dom(\theta_0) = A^\ast \setminus \{ (\neg \phi) \mid \phi \in A^\ast \}$. We then extend $\theta_0$ to a numbering of $A^\ast$ in $\omega$-many steps by defining $\theta_{i+1}$ for each $i  < \omega$:
	\begin{align*}
	\dom(\theta_{i+1}) & = \dom(\theta_i) \cup \{ (\neg \phi) \mid \phi \in \dom(\theta_i) \}\\
	\theta_{i+1}(\phi) & = \theta_i(\phi), \text{ if } \phi \in \dom(\theta_i) \\
	\theta_{i+1}((\neg \phi)) & =
	\begin{cases}
	3 \cdot \langle i+1,j \rangle, & \text{ if } T \vdash \neg \phi \text{ and } \theta_i(\phi)=3 \cdot \langle i,j \rangle +1\\
	3 \cdot \langle i+1,j \rangle + 1, & \text{ if } T \vdash \phi \text{ and } \theta_i(\phi)=3 \cdot \langle i,j \rangle\\
	3 \cdot \langle i+1,j \rangle + 2, & \text{ if } T \not\vdash \phi, T \not\vdash \neg \phi \text{ and } \theta_i(\phi)=3 \cdot \langle i,j \rangle + 2
	\end{cases}
	\end{align*}
	Finally, we set $\theta = \bigcup_{i < \omega} \theta_i$, with $\dom(\theta) = \bigcup_{i < \omega} \dom(\theta_i)$. It is easy to check that $\theta$ is indeed a numbering of $A^\ast$. We now show that ${\mathsf{Pr}(x) \equiv \exists y < x \ x = \overline{3} \times y}$ defines $\{ \theta(\phi) \mid T \vdash \phi \}$. First, we define the height function
	\[
	h^\neg(\phi) = 
	\begin{cases}
	h^\neg(\xi) + 1 & \text{if } \phi \equiv (\neg \xi), \text{ for some expression } \xi,\\
	0 & \text{otherwise.}
	\end{cases}
	\]
	It is then easy to show by induction that for every $T$-theorem $\phi$ the following holds:
	\[
	\theta(\phi) = 3 \cdot \langle h^\neg(\phi),n \rangle \text{ and } \phi \equiv
	\begin{cases}
	\underbrace{(\neg \cdots (\neg}_{h^\neg(\phi)\text{-times}}\chi_n ) \cdots ) & \text{if } h^\neg(\phi) \text{ is even},\\
	\underbrace{(\neg \cdots (\neg}_{h^\neg(\phi)\text{-times}} \mu_n ) \cdots ) & \text{if } h^\neg(\phi) \text{ is odd}.
	\end{cases}
	\]
	Since there is a 1:1-correspondence between $T$-theorems $\phi$ and pairs of natural numbers $\langle h^\neg(\phi),n \rangle$, the set of $\theta$-codes of $T$-theorems is exactly $3 \mathbb{N}$. To show that the $\theta$-tracking function of $\underline{\neg}$ is recursive proceeds as in the case of~$\delta$.
	
	\paragraph{The Numbering $\bm{\eta}$.}
	
	Let $T$ be given as in the construction of the numbering $\delta$. We now construct a numbering $\eta$ of $A^\ast$ such that the tracking functions of the operations $\underline{\mathsf{S}}$, $\underline{+}$, $\underline{\times}$, $\underline{\prime}$, $\underline{=}$, $\underline{\neg}$, $\underline{\wedge}$ and $\underline{\rightarrow}$ are recursive relative to $\eta$ and $\{ \eta(\chi) \mid \mathbb{N} \models \chi \}$ is decidable. As in the case of $\delta$, we partition the natural numbers into decidable infinite sets corresponding to certain syntactic categories:
	{\small
	\begin{align*}
	& & & & \Upsilon^{\mathsf{junk}}(x,y) & = 18 \cdot \langle x,y \rangle +2,\\
	& & & & \Upsilon^{\mathsf{tm}}(x,y) & = 18 \cdot \langle x,y \rangle +5,\\
	\Lambda^{=}(x,y) & = 12 \cdot \langle x,y \rangle, & \Theta^{=}(x,y) & = 12 \cdot \langle x,y \rangle +1, & \Upsilon^{=}(x,y) & = 18 \cdot \langle x,y \rangle +8,\\
	\Lambda^{\neg}(x,y) & = 12 \cdot \langle x,y \rangle +3, & \Theta^{\neg}(x,y) & = 12 \cdot \langle x,y \rangle +4, & \Upsilon^{\neg}(x,y) & = 18 \cdot \langle x,y \rangle +11,\\
	\Lambda^{\wedge}(x,y) & = 12 \cdot \langle x,y \rangle + 6, & \Theta^{\wedge}(x,y) & = 12 \cdot \langle x,y \rangle +7, & \Upsilon^{\wedge}(x,y) & = 18 \cdot \langle x,y \rangle +14,\\
	\Lambda^{\forall}(x,y) & = 12 \cdot \langle x,y \rangle +9, & \Theta^{\forall}(x,y) & = 12 \cdot \langle x,y \rangle +10, & \Upsilon^{\forall}(x,y) & = 18 \cdot \langle x,y \rangle +17.
	\end{align*}
	}
	In what follows we define $\eta$ such that all numbers of true sentences of the form $(s = t)$ are elements of $\Lambda^{=} (\mathbb{N},\mathbb{N}) = \{ \Lambda^{=} (m,n) \mid m,n \in \mathbb{N} \} = 12 \mathbb{N}$, all numbers of true sentences of the form $(\phi \wedge \psi)$ are elements of $\Lambda^{\wedge} (\mathbb{N},\mathbb{N}) = 12 \mathbb{N} + 3$, etc. Since
	\[
	\Lambda^{=} (\mathbb{N},\mathbb{N}) \cup \Lambda^{\neg} (\mathbb{N},\mathbb{N}) \cup \Lambda^{\wedge} (\mathbb{N},\mathbb{N}) \cup \Lambda^{\forall} (\mathbb{N},\mathbb{N}) = 3 \mathbb{N},
	\]
	all ($\eta$-codes of) true sentences are elements of $3 \mathbb{N} = \{ 3 \cdot n \mid n \in \mathbb{N} \}$. Similarly, all ($\eta$-codes of) false sentences are elements of $3 \mathbb{N} +1$, and all ($\eta$-codes of) expressions which are not sentences are elements of $3 \mathbb{N} +2$. Thus $\mathsf{Pr}(x) \equiv \exists y < x \ x = \overline{3} \times y$ satisfies \textnormal{Löb($T, \eta$)} and $T \vdash \neg \mathsf{Pr}(\gn{\bot}{\eta})$, for any sound $T$. Moreover, the set $\{ \eta(\chi) \mid \mathbb{N} \models \chi \}$ is definable by a $\Delta^0_0(\mathsf{exp})$-formula.
	
	In a similar way to the above definitions of $\delta$ and $\theta$, we define $\eta$ by exploiting the fact that the membership of $\neg \phi$ and $\phi \wedge \psi$ in $\{ \chi \mid \mathbb{N} \models \chi \}$, ${\{ \chi \mid \mathbb{N} \not\models \chi \}}$ and ${\{ \chi \mid \chi \text{ is not a sentence} \}}$ is already fully determined by the respective memberships of $\phi$ and $\psi$.\footnote{We use the convention here that $\mathbb{N} \not\models \phi$ implies that $\phi$ is a sentence.} This however does not hold for $\underline{\forall}$, since there are expressions $\phi$, $\psi$ such that both $\mathbb{N} \models \forall x \phi$ and $\mathbb{N} \models \forall x \psi$ but $\phi \in \{ \chi \mid \mathbb{N} \models \chi \}$ and ${\psi \in \{ \chi \mid \chi \text{ is not a sentence} \}}$. Hence, in what follows, $\eta(\forall x \phi)$ is not defined along the lines of $\eta(\neg \phi)$ and $\eta(\phi \wedge \psi)$, resulting in the non-recursiveness of the $\eta$-tracking function of $\underline{\forall}$.
	
	Let $\{ \chi_n \mid n \in \mathbb{N} \}$ be an enumeration without repetitions of true sentences which are of the form $(\forall x \phi)$, let $\{ \nu_n \mid n \in \mathbb{N} \}$ be an enumeration without repetitions of false sentences which are of the form $(\forall x \phi)$, and let $\{ \mu_n \mid n \in \mathbb{N} \}$ be an enumeration without repetitions of formul{\ae} of the form $(\forall x \phi)$ which are not sentences (with $x$ and $\phi$ being  expressions). Moreover, let $\{ \xi_n \mid n \in \mathbb{N} \}$ be an enumeration without repetitions of strings which are not well-formed expressions.
	
	We first define a numbering $\eta_0$ of $\mathit{Term} \cup \{ (\forall x \phi) \mid x \in \mathit{Var}, \phi \in \mathit{Fml} \} $ by setting $\eta_0(s) = \Upsilon^{\mathsf{tm}}(0,\gamma(s))$ for $s \in \mathit{Term}$, and $\eta_0(\chi_n)= \Lambda^{\forall}(0,n) $, $\eta_0(\nu_n)= \Theta^{\forall}(0,n)$ and $\eta_0(\mu_n)= \Upsilon^{\forall}(0,n)$. Moreover, set $\eta^{\mathsf{junk}}(\xi_n)= \Upsilon^{\mathsf{junk}}(0,n) $.
	
	As above we extend $\eta_0$ to a numbering of the well formed expressions, by defining $\eta_{i+1}$ for each $i < \omega$:
	{\small
		\begin{align*}
		\dom(\eta_{i+1}) & = \dom(\eta_i) \cup \{ (s = t) \mid s,t \in \dom(\eta_i) \cap \mathit{Term} \} \cup\\
		& \quad \cup \{ (\neg \phi), (\phi \wedge \psi) \mid \phi, \psi \in \dom(\eta_i) \cap Form\};\\
		\eta_{i+1}(\phi) & = \eta_i(\phi), \text{ if } \phi \in \dom(\eta_i);\\
		\eta_{i+1}((s = t)) & =
		\begin{cases}
		\Lambda^{=}(\eta_i(s),\eta_i(t)), & \text{ if } s,t \in \dom(\eta_i), s,t \text{ are closed terms, } \mathbb{N} \models s=t,\\
		\Theta^{=}(\eta_i(s),\eta_i(t)), & \text{ if } s,t \in \dom(\eta_i),  s,t \text{ are closed terms, } \mathbb{N} \not\models s = t,\\
		\Upsilon^{=}(\eta_i(s),\eta_i(t)), & \text{ if } s,t \in \dom(\eta_i) \cap \mathit{Term},  s \text{ or } t \text{ is not a closed term;}
		\end{cases}\\
		\eta_{i+1}((\neg \phi)) & =
		\begin{cases}
		\Lambda^{\neg}(0,\eta_i(\phi)), & \text{ if } \phi \in \dom(\eta_i), \mathbb{N} \not\models \phi,\\
		\Theta^{\neg}(0,\eta_i(\phi)), & \text{ if } \phi \in \dom(\eta_i),  \mathbb{N} \models \phi,\\
		\Upsilon^{\neg}(0,\eta_i(\phi)), & \text{ if } \phi \in \dom(\eta_i) \cap Form,  \phi \text{ is not a sentence;}
		\end{cases}\\
		\eta_{i+1}((\phi \wedge \psi)) & =
		\begin{cases}
		\Lambda^{\wedge}(\eta_i(\phi),\eta_i(\psi)), & \text{ if } \phi, \psi \in \dom(\eta_i), \mathbb{N} \models \phi, \psi,\\
		\Theta^{\wedge}(\eta_i(\phi),\eta_i(\psi)), & \text{ if } \phi, \psi \in \dom(\eta_i), \mathbb{N} \not\models \phi \text{ or } \mathbb{N} \not\models \psi,\\
		\Upsilon^{\wedge}(\eta_i(\phi),\eta_i(\psi)), & \text{ if } \phi, \psi \in \dom(\eta_i) \cap Form, \phi \text{ or } \psi \text{ is not a sentence.}
		\end{cases}
		\end{align*}
	}
	
	Finally, we define $\eta = \left(  \bigcup_{i < \omega} \eta_i \right) \cup \eta^{\mathsf{junk}}$, with $\dom(\eta) = \left( \bigcup_{i < \omega} \dom(\eta_i) \right) \cup \mathit{Junk}$. Showing that $\eta$ is a numbering of $A^\ast$ such that $\{ \eta(\chi) \mid \mathbb{N} \models \chi \}$ can be defined by a $\Delta^0_0(\mathsf{exp})$-formula and the $\eta$-tracking functions of all operations operations $\underline{\mathsf{S}}$, $\underline{+}$, $\underline{\times}$, $\underline{\prime}$, $\underline{=}$, $\underline{\neg}$, $\underline{\wedge}$ and $\underline{\rightarrow}$ are recursive proceeds as in the case of $\delta$ above.
	
	\paragraph{The Numbering $\bm{\zeta}$.}
	
	We conclude with a construction of a numbering $\zeta$ of $A^\ast$ which is monotonic and such that $\zeta(T^\vdash) = \{ n \mid \Exists m \geq 10 \ n = m^3 \}$. Hence the formula ${\mathsf{Pr}(x) \equiv \exists y < x (x = y \times y \times y \wedge  y \geq \overline{10})}$ is a $\Delta^0_0$-binumeration of ($\zeta$-codes) of $T$-theorems.
	
	Let $\alpha_0 \equiv ( \mathsf{0} = \mathsf{0} )$ and $\alpha_{i+1} \equiv  (  \alpha_i \wedge (\mathsf{0} = \mathsf{0}))$, for $i < \omega$. Let $(\beta_i)_{i < \omega}$ be an enumeration (without repetitions) of the substrings of elements of $(\alpha_i)_{i < \omega}$ which are themselves not elements of $(\alpha_i)_{i < \omega}$. We assume that $(\beta_i)_{i < \omega}$ is given in length-first ordering, i.e., $i < j$ iff the length of $\beta_i$ is smaller than the length of $\beta_j$, or $\beta_i$ and $\beta_j$ have the same length and $\beta_i$ is lexicographically (relative to some total order on $A$) smaller than $\beta_j$. Moreover, let $\gamma_0 \equiv \mlq \prime \mrq$ and $\gamma_{i+1} \equiv \gamma_i \ast \mlq \prime \mrq$, for $i < \omega$. 
	
	Let $(\chi_i)_{i < \omega}$ be a monotonic enumeration of ${A^\ast \setminus ( (\alpha_i)_{i < \omega} \cup (\beta_i)_{i < \omega} \cup (\gamma_i)_{i < \omega} )}$ without repetitions. We thus have $j \leq i$ if $\chi_j$ is a subexpression of $\chi_i$. This can for instance be achieved by starting with a surjective and monotonic standard numbering of~$A^\ast$.
	
	We define a bijection $g \colon \mathbb{N} \to A^\ast$ by recursion as follows. Let $\prec$ denote the strict substring relation on $A^\ast$. Let $n \in \mathbb{N}$ and assume that $g(m)$ is already defined for all $m < n$. Let $(\mu_i)_{i< \omega}$ be any given enumeration of strings. For the sake of brevity, we also write $\widetilde{\mu}$ instead of $(\mu_i)_{i < \omega}$. We first set
	\[
	\mathfrak{m}(n,\widetilde{\mu}) := \min \{ k \mid \Forall m < n \ g(m) \not\equiv \mu_k  \}.
	\]
	That is, for a given enumeration $\widetilde{\mu}$ and input $n$, the function $\mathfrak{m}(\cdot,\widetilde{\mu}) \colon \mathbb{N} \to \mathbb{N}$ outputs the smallest index $k$ such that $\mu_k$ is not the $g$-value of any $m < n$. We now define
	\begin{align*}
	g(n) \equiv
	\begin{cases}
	\begin{cases}
	\chi_{\mathfrak{m}(n,\widetilde{\chi})} & \text{if }  T \vdash \chi_{\mathfrak{m}(n,\widetilde{\chi})} \text{ and}\\
	&\Forall \psi \prec \chi_{\mathfrak{m}(n,\widetilde{\chi})} \Exists k < n\ g(k) = \psi \\
	\alpha_{\mathfrak{m}(n,\widetilde{\alpha})} & \text{otherwise}
	\end{cases}
	& \text{if }  \Exists m \geq 10 \ n = m^3; \\
	\begin{cases}
	\chi_{\mathfrak{m}(n,\widetilde{\chi})} & \text{if }  T \not\vdash \chi_{\mathfrak{m}(n,\widetilde{\chi})} \text{ and}\\
	& \Forall \psi \prec \chi_{\mathfrak{m}(n,\widetilde{\chi})} \Exists k < n\ g(k) = \psi \\
	\gamma_{\mathfrak{m}(n,\widetilde{\gamma})} & \text{otherwise}
	\end{cases}
	& \text{if }  \Exists m \geq 10 \ n = m^3+1; \\
	\begin{cases}
	\beta_{\mathfrak{m}(n,\widetilde{\beta})} & \text{if } \Forall \psi \prec \beta_{\mathfrak{m}(n,\widetilde{\beta})} \Exists k < n\ g(k) = \psi \\
	\gamma_{\mathfrak{m}(n,\widetilde{\gamma})} & \text{otherwise}
	\end{cases}
	& \text{otherwise}.\\
	\end{cases}
	\end{align*}
	Since $\widetilde{\alpha}$, $\widetilde{\beta}$, $\widetilde{\gamma}$ and $\widetilde{\chi}$ are pairwise disjoint enumerations without repetitions, $g$ is injective. We now define $\zeta := g^{-1}$.
	
	We show that $\zeta$ is indeed a function with domain $A^*$, i.e., we have to show that $g$ is surjective. Informally speaking, this is tantamount to showing that there is no number from which on the values of $g$ only belong to three or less of the four enumerations.
	
	We first show that for each $i$ there exists $n$ such that $g(n) \equiv \alpha_i$. Assume that this is not the case, and let $i_0 = \max \{i \mid \exists n \ g(n) \equiv \alpha_i \}$. Then there exists an index $j_0$ such that for all $j > j_0$:
	\begin{equation}
	\label{equ:app:mon}
	\text{If } T \vdash \chi_{j} \text{ then } \Forall \psi \prec \chi_{j} \Exists k < n\ g(k) \equiv \psi.
	\end{equation}
	
	Let now $i > i_0$ be given such that $(\alpha_{i} \vee \alpha_{i}) \equiv \chi_j$ for some $j > j_0$ (note that such numbers exists since there are infinitely many strings of the form $(\alpha_{i} \vee \alpha_{i})$). Then $T \vdash (\alpha_{i} \vee \alpha_{i})$ and $\alpha_{i} \prec (\alpha_{i} \vee \alpha_{i})$, but since $i > i_0$ there is no $k$ such that $g(\alpha_{i}) \equiv k$, in contradiction to (\ref{equ:app:mon}).
	
	Analogously it can be shown that for each $i$ there exists $n$ such that $g(n) \equiv \gamma_i$ (for instance, by considering $\mlq \mathsf{v} \mrq \ast \gamma_i$ instead of $(\alpha_{i} \vee \alpha_{i})$).
	
	We now show that each $\beta_i$ is the $g$-value of some number, by induction over $i$. Since $\widetilde{\beta}$ is length-first-ordered, $\beta_0$ has no proper substrings and hence $g(0) \equiv \beta_0$. Let $i > 0$, and assume that each $\beta_j$ with $j <i$ is the $g$-value of some number. If $\psi \prec \beta_i$, then $\psi \equiv \beta_j$ for some $j < i$ or $\psi \equiv \alpha_l$ for some $l \in \mathbb{N}$. Since by induction hypothesis each $\beta_j$ with $j < i$ is the $g$-value of a number, and we have seen above that each $\alpha_l$ is the $g$-value of a number, there exists a smallest number $n$ such that $n$ is not of the form $m^3$ or $m^3+1$ (for $m \geq 10$), and the $g$-value of each (proper) substring of $\beta_i$ is smaller than $n$. Then $g(n) \equiv \beta_{\mathfrak{m}(n,\widetilde{\beta})} \equiv \beta_i$.
	
	We now show that each $\chi_i$ is the $g$-value of some number, by induction over $i$. Since $\widetilde{\chi}$ is monotonic, $\chi_0$ is a string of length $1$. Hence $g(10^3+1) \equiv \chi_0$.
	
	Let $i > 0$ be given and assume that each $\chi_j$ with $j <i$ is the $g$-value of some number. If $\psi \prec \chi_i$, then either $\psi \equiv \chi_j$ with $j < i$, or $\psi \equiv \alpha_l$ or $\psi \equiv \beta_l$ or $\psi \equiv \gamma_l$, for some $l \in \mathbb{N}$. Once again, by induction hypothesis each $\chi_j$ with $j < i$ is the $g$-value of a number, and we have seen above that each $\alpha_l$, $\beta_l$ and $\gamma_l$ is the $g$-value of a number. Hence $g(n) \equiv \chi_{\mathfrak{m}(n,\widetilde{\chi})} \equiv \chi_i$, where $n$ is the smallest number of the appropriate form, i.e. $m^3$ in case that $T \vdash \chi_i$, and $m^3+1$ otherwise (for $m \geq 10$), such that the $g$-value of each (proper) substring of $\chi_i$ is smaller than $n$.
	
	It remains to show that $\zeta$ is monotonic. Let $\phi \prec \psi$ be given strings.
	
	If $\psi \equiv \gamma_i$, for some $i \in \mathbb{N}$. Then $\phi \equiv \gamma_j$, for some $j < i$. It follows immediately from the definition of $g$ that if $\zeta(\gamma_i )=m $, there exists $k < m$ such that $g(k) \equiv \gamma_j$. Hence $\zeta(\gamma_j) < \zeta(\gamma_i)$.
	
	If $\psi \equiv \beta_i$ or $\psi \equiv \chi_i$, for some $i \in \mathbb{N}$, the claim immediately follows from the defining clauses of $g$.
	
	If $\psi \equiv \alpha_i$, for some $i \in \mathbb{N}$, then $\phi \equiv \alpha_j$, for some $j < i$, or $\phi \equiv \beta_l$ for some $l \in \mathbb{N}$. By definition of $g$, we have $\zeta(\alpha_j) < \zeta(\alpha_i)$.
	It remains to show that for all $i,l \in \mathbb{N}$:
	\begin{equation*}
	(B[i,l]) \quad \text{If } \beta_l \prec \alpha_i \text{ then } g^{-1}(\beta_l) < g^{-1}(\alpha_i).
	\end{equation*}
	We first show that $A[i]$ implies $\Forall l \, B[i,l]$, for any $i \in \mathbb{N}$; where $A[i]$ is the statement:
	\begin{equation*}
	\text{If } g(m^3) \equiv \alpha_i \text{ for some $m$, then } \Exists n, l \ (m-1)^3+1 < n < m^3 \Land g(n) \equiv \gamma_l.
	\end{equation*}
	
	\enquote{$\Forall i (A[i] \Rightarrow \Forall l B[i,l])$}: Let $g(m^3) \equiv \alpha_i$ for some $m$. By $A[i]$, we find $n$ and $l$ such that $g(n) \equiv \gamma_l$ and $(m-1)^3+1 < n < m^3$.
	Since $n$ cannot be of the form $s^3$ or $s^3+1$, it follows that $g(n)$ is obtained by the second case of the third clause of the definition of $g$. Hence, there exists $\psi \prec \beta_{\mathfrak{m}(n,\widetilde{\beta})}$ such that ($\ast$) $g(k) \not\equiv \psi$ for all $k < n$. Clearly, $\psi$ is not contained in $\widetilde{\beta}$. For otherwise, $\psi \equiv \beta_j$ for some $j \in \mathbb{N}$. Then $j < \beta_{\mathfrak{m}(n,\widetilde{\beta})}$ by definition of $\widetilde{\beta}$. But since $\beta_j$ is not the $g$-value of any number smaller than $n$, we have $\mathfrak{m}(n,\widetilde{\beta}) \leq j$, resulting in a contradiction.
	
	Hence, $\psi \equiv \alpha_j$ for some $j \in \mathbb{N}$. Since $g^{-1}(\alpha_i) \equiv m^3$, we have $g^{-1}(\alpha_{j'}) \leq (m-1)^3$ for every $j' < i$. By ($\ast$), $\alpha_j$ is not the $g$-value of any number smaller than $n$. Since $(m-1)^3 < n$ we thus conclude that $i \leq j$. Therefore $\alpha_i \preceq \alpha_j \prec \beta_{\mathfrak{m}(n,\widetilde{\beta})}$.
	
	We show now that $B[i,l]$ holds, for any given $l$. If $\beta_l \prec \alpha_i$, then also $\beta_l \prec \beta_{\mathfrak{m}(n,\widetilde{\beta})}$. Hence $l < \mathfrak{m}(n,\widetilde{\beta})$, by definition of $\widetilde{\beta}$. But then there exists $k < n$ such that $g(k) \equiv \beta_l$. Since $k < n < m^3$ we conclude that $g^{-1}(\beta_l) < g^{-1}(\alpha_i)$, as desired.
	
	\enquote{$\Forall i A[i]$}: We show now by induction that $A[i]$ holds for all $i \in \mathbb{N}$.
	
	Since $\alpha_0$ has length $5$, it has not more than $\frac{5 (5+1)}{2}$ substrings. By definition of $g$ and $\widetilde{\beta}$, it is easy to check that there is a number $n < 10^3$ such that the $g$-value of each proper substring of $\alpha_0$ is smaller than $n$. Hence $\alpha_0 \prec \beta_{\mathfrak{m}(n,\widetilde{\beta})}$. But since $g^{-1}(\alpha_0) \geq 10^3$ we have $g(n) \equiv \gamma_{\mathfrak{m}(n,\widetilde{\beta})}$.
	
	Let now $i > 0$ be given, and assume that $A[i-1]$ holds. Let $g^{-1}(\alpha_i) = m^3$. We assume that $A[i]$ does not hold. Let $n$ be fixed such that $(m-1)^3+1 < n < m^3$. We then have $g(n) \equiv \beta_{l}$ for some $l$. Hence,
	\begin{equation}
	\label{equation:proofmonoton}
	\Forall \psi \prec \beta_{l} \Exists k < n\ g(k) \equiv \psi.
	\end{equation}
	If $\beta_{l} \prec \alpha_{i-1}$, then $g^{-1}(\beta_l) < g^{-1}(\alpha_{i-1}) \leq (m-1)^3$ by $B[i-1,l]$ (using the fact that $\Forall i (A[i] \Rightarrow \Forall l B[i,l])$ and the induction hypothesis $A[i-1]$). But since $(m-1)^3 < n = g^{-1}(\beta_l)$, we conclude that $\beta_{l} \prec \alpha_{i-1}$ does not hold. We now show that $\beta_{l} \prec \alpha_{i}$. Assume this is not the case, then $\alpha_{i} \prec \beta_{l}$ by definition of $\widetilde{\beta}$. Using (\ref{equation:proofmonoton}) we can then derive the contradiction $g^{-1}(\alpha_i) < n < m^3 = g^{-1}(\alpha_i)$.
	
	We thus have shown that $\beta_{l} \prec \alpha_{i}$, but $\beta_{l} \not\prec \alpha_{i-1}$. Since $\beta_{l}$ is the $g$-value of any given number between $(m-1)^3+1$ and $m^3$, there are $\Delta$-many substrings of $\alpha_l$ which are not substrings of $\alpha_{i-1}$, with $\Delta := (m^3 - ((m-1)^3+1) -1$.
	
	We derive a contradiction by showing that 
	\begin{equation}
	\label{equation:proofmonoton2}
	\#(\alpha_{i}) - \#(\alpha_{i-1}) < \Delta,
	\end{equation}
	where $\#(s)$ denotes the number of substrings of $s$.  Let $\mathit{lh}(s)$ denote the length of~$s$. Since $\mathit{lh}(\alpha_{i}) = 5 + 8i$, it is easy to check that
	\[
	\#(\alpha_{i}) - \#(\alpha_{i-1}) \leq 8 \mathit{lh}(\alpha_{i}) = 64i + 40.
	\]
	Moreover $i \leq m - 10$, since $(i + 10)^3 \leq g^{-1}(\alpha_i) = m^3$. Hence 
	\[
	\#(\alpha_{i}) - \#(\alpha_{i-1}) \leq 8 \mathit{lh}(\alpha_{i}) = 64i + 40 \leq 64(m - 10) + 40 = 64m - 600
	\]
	It remains to check that $64m - 600 < \Delta$. But this is equivalent to
	\[
	64m - 600 < (m^3 - ((m-1)^3+1) -1 = 3m^2 - 3m -1,
	\]
	which in turn is equivalent to the true equation
	\[
	0 < 3m^2 - 67m +599.
	\]
	This concludes the proof of (\ref{equation:proofmonoton2}), in contradiction to there being $\Delta$-many substrings of $\alpha_l$ which are not substrings of $\alpha_{i-1}$. Thus $A[i]$ obtains.
	
	\section{Acknowledgements} I am indebted to Arnon Avron, Volker Halbach, Karl-Georg Niebergall, Lavinia Picollo, Dan Waxman and two anonymous referees for extensive and extremely helpful comments on earlier versions of this paper. Special thanks go to Albert Visser for his encouragement and many fruitful discussions on numberings over the last years, which greatly improved this paper. I would also like to thank Udi Boker, Luca Castaldo, Nachum Dershowitz, Mirko Engler, Michael Goldboim, David Kashtan, Ran Lanzet, Carlo Nicolai,  Fedor Pakhomov, Carl Posy, Lorenzo Rossi and Gil Sagi for helpful conversations on topics presented in this paper. Finally, I am grateful to the audiences of my talks on this material for their valuable feedback. This research was supported by the Minerva Foundation of the Max Planck Society.
	
	\setlength{\bibsep}{0.0pt}
	\bibliographystyle{rsl}
	\bibliography{mybib}{}
	
\end{document}